\documentclass[final, 11pt,leqno,a4paper]{amsart}
\usepackage{amssymb,enumerate,color}
	\usepackage[all]{xy}   
\usepackage{amsthm}      
\usepackage{indentfirst}
\usepackage{paralist}
\usepackage{amsmath}
\usepackage{mathrsfs}
\usepackage{geometry}
\theoremstyle{theorem}
\newtheorem{mainthm}{Theorem}
\newtheorem{thm}{Theorem}[section]
\newtheorem{lem}[thm]{Lemma}
\newtheorem{prop}[thm]{Proposition}
\newtheorem{cor}[thm]{Corollary}
\newtheorem{conj}[mainthm]{Conjecture}
\theoremstyle{definition}
\newtheorem{defn}[thm]{Definition}
\newtheorem{hyp}[thm]{Hypothesis}
\newtheorem{rmk}[thm]{Remark}

\numberwithin{equation}{section}

\newcommand{\IBr}{\operatorname{IBr}\nolimits}
\newcommand{\Irr}{\operatorname{Irr}\nolimits}
\newcommand{\PSL}{\operatorname{PSL}\nolimits}
\newcommand{\PSU}{\operatorname{PSU}\nolimits}
\newcommand{\GL}{\operatorname{GL}\nolimits}
\newcommand{\GU}{\operatorname{GU}\nolimits}
\newcommand{\SL}{\operatorname{SL}\nolimits}
\newcommand{\SU}{\operatorname{SU}\nolimits}

\newcommand{\Sp}{\operatorname{Sp}\nolimits}
\newcommand{\Spin}{\operatorname{Spin}\nolimits}
\newcommand{\CSp}{\operatorname{CSp}\nolimits}

\newcommand{\Aut}{\operatorname{Aut}\nolimits}

\newcommand{\Res}{\operatorname{Res}\nolimits}
\newcommand{\Ind}{\operatorname{Ind}\nolimits}

\newcommand{\J}{\operatorname{J}}
\newcommand{\ZZ}{\operatorname{Z}}
\newcommand{\Dec}{\operatorname{Dec}}

\newcommand{\F}{\mathbb{F}}
\newcommand{\barF}{\overline{\mathbb{F}}}
\newcommand{\Z}{\mathbb{Z}}
\newcommand{\cF}{\mathcal{F}}

\newcommand{\bG}{\mathbf{G}}
\newcommand{\tG}{\widetilde{G}}

\newcommand{\ts}{\tilde{s}}
\newcommand{\tH}{\widetilde{H}}
\newcommand{\tL}{\widetilde{L}}
\newcommand{\tC}{\widetilde{C}}
\newcommand{\C}{\mathrm{C}}
\newcommand{\tbG}{\widetilde{\mathbf{G}}}
\newcommand{\tbL}{\widetilde{\mathbf{L}}}
\newcommand{\tcC}{\widetilde{\mathcal{C}}}
\newcommand{\tcB}{\widetilde{\mathcal{B}}}
\newcommand{\bL}{\mathbf{L}}
\newcommand{\bN}{\mathbf{N}}
\newcommand{\bT}{\mathbf{T}}
\newcommand{\bH}{\mathbf{H}}

\newcommand{\cE}{\mathcal{E}}
\newcommand{\cB}{\mathcal{B}}
\newcommand{\cC}{\mathcal{C}}

\newcommand{\tchi}{\widetilde{\chi}}
\newcommand{\tphi}{\widetilde{\phi}}

\newcommand{\sC}{\mathscr{C}}
\newcommand{\cO}{\mathcal{O}}
\newcommand{\scU}{\mathscr{U}}
\newcommand{\tscU}{\widetilde{\mathscr{U}}}
\newcommand{\bm}{\mathbf{m}}

\newcommand{\eps}{\epsilon}
\newcommand{\vare}{\varepsilon}
\newcommand{\la}{\lambda}
\newcommand{\ti}{\times}

\begin{document}

\title[Unitriangular basic sets and coprime actions]{Unitriangular basic sets, Brauer characters\\ and coprime actions}

\date{\today}
\author{Zhicheng Feng}
\address{School of Mathematics and Physics, University of Science and Technology Beijing, Beijing 100083, China\\
{\rm Current address}: SICM and Department of Mathematics, Southern University of Science and Technology, Shenzhen 518055, China}
\email{fengzc@sustech.edu.cn}

\author{Britta Sp\"ath}
\address{School of Mathematics and Natural Sciences, University of Wuppertal, Gau{\ss}strasse 20, 42119 Wuppertal, Germany}
\email{bspaeth@uni-wuppertal.de}

\thanks{The first author gratefully acknowledges financial support by NSFC (11901028 and 11631001). The research of the second author was conducted in the framework of the 
	research training group	\emph{GRK 2240: Algebro-Geometric Methods in Algebra, Arithmetic and 
		Topology}, funded by the DFG}



\maketitle

\begin{abstract} 	We show that the decomposition matrix of a given group $G$ is unitriangular, whenever $G$ has a normal subgroup $N$ such that the decomposition matrix of $N$  is unitriangular, $G/N$ is abelian  and certain characters of $N$ extend to their stabilizer in $G$. Using the recent result by Brunat--Dudas--Taylor establishing that unipotent blocks have a unitriangular decomposition matrix, this allows us to prove that  blocks of groups of  quasi-simple groups of Lie type have a unitriangular decomposition matrix, whenever they are related via Bonnaf\'e--Dat--Rouquier's equivalence to a unipotent block.  This is then applied to study the action of automorphisms on Brauer characters of finite quasi-simple groups. We use it to verify the  so-called {\it  inductive Brauer--Glaubermann condition}, that aims to establish a Glauberman correspondence for Brauer characters, given a coprime action.  
\end{abstract}

\section{Introduction}
Irreducible representations of finite groups over fields of positive characteristic $\ell$ make a quite difficult subject. Many important questions can nevertheless be studied via the ($\ell$-){\it decomposition matrix} relating for a given finite group $G$ the set $\IBr(G)$ of irreducible characters in characteristic $\ell$ to the set $\Irr(G)$ of better understood ordinary characters over $\Bbb C$. The entries of the decomposition matrix are indexed by $\Irr(G)\times \IBr(G)$. Many decomposition matrices of finite groups enjoy the property of so-called {\it unitriangularity}. Unitriangularity implies the existence of an injective map $$\IBr(G)\hookrightarrow \Irr(G)$$ whose image $\cB\subseteq \Irr(G)$ is a so-called basic set (see Definition~\ref{DefBasic} below) and such that the square submatrix of the $\ell$-decomposition matrix corresponding to $\cB\times \IBr(G)$ is lower unitriangular for some ordering of the rows and columns. This property seems to be shared by most finite simple groups and their central extensions except possibly groups of Lie type of characteristic $\ell$. Unitriangularity of $\ell$-decomposition matrices has been known for a long time for symmetric groups (see \cite[\S 7.1]{JK81}), general linear groups of characteristic $\not=\ell$ (Dipper \cite{D85a,D85b}), unitary groups (Geck \cite{Ge91}), and many other finite groups of Lie type for certain primes $\ell$. Geck (see \cite[p. 5]{G90}) made the following conjecture

\begin{conj}\label{GeckC} {\rm (Geck 1990)}
	Let $\ell$ be a prime. If $G$ is a finite group of Lie type of characteristic $\not=\ell$, then its $\ell$-decomposition matrix is unitriangular.
\end{conj}

This question can of course be split into the various $\ell$-blocks of $G$.
A recent breakthrough was the proof by Brunat--Dudas--Taylor of the unitriangularity of the unipotent $\ell$-blocks of finite groups of Lie type under mild restrictions on $\ell$, see \cite[Thm.~A]{BDT20}. The basic set is provided in that case by the set of unipotent characters.

Suppose that $\bG$ is a connected reductive group with an $\F_q$-structure given by a Steinberg endomorphism $F : \bG \to \bG$ where $\ell\nmid q$.
According to a result of Bonnaf\'e--Dat--Rouquier \cite[Thm.~1.1]{BDR17},
most $\ell$-blocks of $\bG^F$
are Morita equivalent to a block of a group $\bN^F$ which covers a unipotent block of a normal subgroup $\bL^F$, where $\bL$ is an $F$-stable Levi subgroup of $\bG$ normal in some $F$-stable subgroup $\bN$, with abelian quotient group $\bN^F /\bL^F$ .

Since Morita equivalences over an adequate local ring preserve decomposition matrices, this leads naturally to the question of extending the unitriangularity property of decomposition matrices from $\bL^F$ to $\bN^F$.
Our Theorem \ref{TriangUp} below shows such a ``going-up" property for unitriangular basic sets in the general situation of a normal inclusion of finite groups $L\unlhd N$ with abelian quotient. 
Applying this in connection with the result of Bonnaf\'e--Dat--Rouquier, we can show that the unitriangularity of unipotent blocks proven by Brunat--Dudas--Taylor implies that many more non-unipotent blocks of finite reductive groups have a unitriangular decomposition matrix. For good primes $\ell$, this essentially reduces the question to so-called {\it 
isolated} non-unipotent $\ell$-blocks. 

Unitriangularity of basic sets is a key asset to study the action of automorphisms and subsequent questions of extendibility of elements of $\IBr(G)$ for $G$ an abstract finite group. This are crucial questions to tackle counting conjectures on modular characters through the recent reduction theorems leading to related questions on finite quasi-simple groups. We focus in this paper on a conjecture due to Gabriel Navarro (see \cite{Na94}).
If $A$ is a finite groups acting coprimely on $G$ (i.e., $\gcd(|A|,|G|)=1$) via automorphisms, the \emph{Glauberman--Isaacs correspondence} is a bijection $$\Irr_A(G)\to \Irr(\C_G(A))$$ between the set of irreducible $A$-invariant characters
of $G$ and the set of irreducible characters of the subgroup $\C_G(A)$, see \cite[\S 13]{Is76}. Navarro's question asks for replacing $\Irr$ by $\IBr$ in the above.

\begin{conj}\label{conj-corr} {\rm (Navarro 1994)}
	Let $\ell$ be a prime. Suppose that $A$ acts coprimely via automorphisms on $G$. Then
	$$|\IBr_A(G)|=|\IBr(\C_G(A))|,$$ where
	$\IBr_A(G)$ is the set of irreducible $\ell$-Brauer characters of $G$ fixed by $A$.
\end{conj}

Brauer's argument on the character table~implies Conjecture~\ref{conj-corr} whenever $A$ acts by a cyclic group
(see for instance~\cite[Thm.~3.1]{NST17}).
Using the classification of the finite simple groups we then get that
Conjecture~\ref{conj-corr} holds for any quasi-simple group $G$.

In~2016, the second author and Vallejo~\cite{SV16} reduced Conjecture~\ref{conj-corr} to a problem on quasi-simple groups, proving that if $A$ acts coprimely on $G$, and every non-abelian simple group involved in $G$ satisfies the so-called \emph{inductive Brauer--Glauberman (iBG) condition} (see Definition~\ref{iBGC}), then $G$ and $A$ satisfy Conjecture~\ref{conj-corr}. The (iBG) condition was verified by
Navarro, the second author and Tiep~\cite{NST17} for certain families of simple groups, including  simple groups not of Lie type,  simple groups of Lie type in defining characteristic and simple groups with cyclic outer automorphism groups.
In addition,  Farrell and Ruhstorfer~\cite{FR21} proved a part of the (iBG) condition, called the \emph{fake Galois actions}, for all simple groups.

In this paper, we consider the (iBG) condition for simple groups of Lie type in non-defining characteristic and prove the following.
Suppose that $\bG$ is a simple simply connected algebraic group and $F : \bG \to \bG$ a Steinberg endomorphism endowing $\bG$ with an $\F_q$-structure. Let $\widetilde \bG$ be a regular embedding of $\bG$ and
we extend $F$ to be a Steinberg endomorphism of $\widetilde \bG$ (see \cite[1.7.5]{GM20}).

\begin{mainthm}\label{main-thm1}
	If both $\widetilde \bG^F$ and $\bG^F$ satisfy Conjecture~\ref{GeckC} for any prime $\ell\nmid q$ through maps compatible with linear characters and field automorphisms (see Hypothesis~\ref{unitri}) whenever $\mathrm{Out}(\bG^F)$ is non-cyclic, then Conjecture~\ref{conj-corr} is true.
\end{mainthm}

Therefore, we verify the (iBG) condition assuming Conjecture~\ref{GeckC}.
As a consequence, we prove (iBG) for simple groups of Lie type~$\mathsf A$ and any prime~$\ell$ (see Theorem \ref{for-type-A}). We also obtain (iBG) condition for simple groups of type $\mathsf C$ and the prime~2 (see Theorem~\ref{typeC-2}), leading also to 

\medskip\noindent {\bf Corollary 4.} {\it
	If $G$ has abelian Sylow $3$-subgroups,
	then  Conjecture~\ref{conj-corr}  holds for any $A$ acting coprimely on $G$ and any prime $\ell$.}

\medskip

The structure of this paper is as follows.
In \S \ref{Preliminiaries} we give some notation and preliminaries, recalling some basic statements on coprime actions and Brauer characters. In \S \ref{sec-decom-CLI} we establish several ``going-up" properties for unitriangular basic sets.
Then we recall the definition of the (iBG) condition from~\cite{SV16} and give a criterion for the (iBG) condition in \S\ref{sec-iBGC}.
The (iBG) condition for simple groups of Lie type in non-defining characteristic is reduced to a question on untriangular basic sets in \S\ref{SG-Lie} and Theorem~\ref{main-thm1} and \ref{for-type-A} are also proven there. 
In  \S\ref{application}, we give an application of Theorem~\ref{main-thm1} and prove Corollary 4. 
Finally, in  \S\ref{sec-unitri-Lie} we show that a lot more blocks of finite reductive groups than the ones treated in \cite{BDT20} satisfy Conjecture~\ref{GeckC} (see Theorem~\ref{thm_appl}).

\section{Preliminaries}\label{Preliminiaries}

\subsection{Some notation} The notation for ordinary characters mostly follows~\cite{Is76} while the notation for Brauer characters mostly follows~\cite{Na98}.
All Brauer characters are considered for the fixed prime $\ell$.
For a given finite group $G$, we denote by $\IBr(G)$ the set of irreducible ($\ell$-)Brauer characters of $G$ seen as $G$-central $\mathbb C$-valued functions on the set $G_{\ell '}$ of $\ell$-regular elements of $G$.
We denote the restriction of $\chi\in \Irr(G)\cup\IBr(G)$ to a subgroup $H\le G$ by $\Res^G_H(\chi)$. The set of irreducible components of $\Res^G_H(\chi)$ is denoted by $\Irr (H\mid\chi)$ or $\IBr (H\mid\chi)$ according to $\chi$ being an ordinary or a Brauer character. Whenever $\psi\in \Irr(H)\cup\IBr(H)$, one denotes by $\Ind_H^G(\psi)$ the induced character. Similarly its set of irreducible components is denoted by $\Irr(G\mid\psi)$ or $\IBr(G\mid\psi)$ according to $\psi\in\Irr(H)$ or $\psi\in\IBr(H)$.

Assume now $N\unlhd G$. We sometimes identify the (Brauer) characters of $G/N$ with the (Brauer) characters of $G$ whose kernel contains $N$. 

For subsets $\mathcal N\subseteq \IBr(N)$ and $\mathcal G\subseteq \IBr(G)$, we define $$\IBr(G\mid\mathcal N ):= \bigcup\limits_{\chi\in \mathcal N }\IBr(G\mid \chi)\text{  and  }
\IBr(N\mid\mathcal G):= \bigcup\limits_{\psi\in \mathcal G}\IBr(N\mid \psi).$$
Similarly, if $\mathcal N\subseteq \Irr(N)$ and $\mathcal G\subseteq \Irr(G)$, then we also define the set $\Irr(G\mid\mathcal N )$
and $\Irr(N\mid\mathcal G)$ analogously.

If a group $A$ acts on a finite set $X$, we denote by $A_x$ the stabiliser of $x\in X$ in $A$.
If $A$ acts on a finite group $G$ via automorphisms, there is a natural action of $A$ on $\Irr(G)$ (or $\IBr(G)$, resp.) given by ${}^{a^{-1}}\chi(g)=\chi^a(g)=\chi({}^ag{})$ for every $g\in G$, $a\in A$ and $\chi\in \Irr(G)$ (or  $\chi\in \IBr(G)$, resp.).
For $B\le A$, we denote by $\IBr_{B}(G)$ the set of $B$-invariant irreducible Brauer characters of $G$.
For $\chi\in \IBr(G)$ and $\psi\in\IBr(N)$,
we define
$\IBr_{B}(N\mid \chi)=\IBr(N\mid \chi)\cap \IBr_{B}(N),$
and $\IBr_{B}(G\mid \psi)=\IBr(G\mid \psi)\cap \IBr_{B}(G).$
Similarly,  we denote by $\Irr_{B}(G)$ the set of $B$-invariant irreducible characters of $G$.
Suppose that $N\unlhd G$ is stable under the action of $B$.
For $\chi\in \Irr(G)$ and $\psi\in\Irr(N)$,
we define
$\Irr_{B}(N\mid \chi)=\Irr(N\mid \chi)\cap \Irr_{B}(N),$
and $\Irr_{B}(G\mid \psi)=\Irr(G\mid \psi)\cap \Irr_{B}(G).$

\subsection{Clifford theory with abelian quotient} Suppose $N\unlhd G$ are finite groups with abelian $G/N$.
Let $\chi\in\Irr(G)$ (or $\chi\in\IBr(G)$, resp.) and $\theta\in\Irr(N\mid\chi)$ (or $\theta\in\IBr(N\mid\chi)$, resp.).
In the case where $\theta\in\Irr(N)$, then $\Irr(G/N)$  acts on $\Irr(G\mid\theta)$  by multiplication and
$\Irr(G\mid\theta)$ is an $\Irr(G/N)$-orbit.
Moreover, if $\theta$ extends to $G_\theta$ then
\begin{equation}\label{sta-irr}
G_\theta= \bigcap\limits_{\{\lambda\in\Irr(G/N)\ |\ \chi\lambda=\chi\}} \ker(\lambda) ,
\addtocounter{thm}{1}\tag{\thethm}
\end{equation} as can be seen from Gallagher's theorem (\cite[6.17]{Is76}) and the uniqueness in Clifford correspondence.
In the case when $\theta\in\IBr(N)$ extends to $G_\theta$, 
then $\IBr(G/N)$ acts transitively on $\IBr(G\mid\theta)$ by multiplication (this is Problem 6.2 of \cite{Is76} for Brauer characters).
We define $\J^G_N(\chi)$ as the smallest subgroup of $G$ containing $G_\theta $ and such that $G/\J^G_N(\chi)$ is an $\ell$-group.
According to~\cite[Lemma~2.14]{BS20}
\begin{equation}\label{sta-ibr}
G_\theta=\bigg( \bigcap\limits_{\{\lambda\in\IBr(G/N)\ |\ \chi\lambda=\chi\}} \ker(\lambda)   \bigg)\cap \J^G_N(\chi).
\addtocounter{thm}{1}\tag{\thethm}
\end{equation}

In this paper we will mainly be concerned with normal inclusions $N\unlhd G$ such that $G/N$ is abelian and all elements of $\Irr(N)\cup\IBr(N)$ extend to their stabilizer in $G$. It is easy to see from Clifford theory that this last hypothesis is equivalent to $\Res^G_N(\chi)$ being multiplicity-free for every $\chi\in \Irr(G)\cup\IBr(G)$.

\subsection{Coprime actions} We gather here some more technical statements related with the representation theory of a finite group acted upon by another finite group when the two groups have coprime orders.

\begin{prop}\label{rela-G-CN}
Let $A$ act coprimely via automorphisms on $G$, where $A$ and $G$ are finite groups, and let $N\unlhd G$ be $A$-stable.
Let $C=\C_G(A)$. 
\begin{enumerate}[\rm(i)]
	\item Let $\chi\in\IBr_A(G)$. Then $\Res_N^G(\chi)$  has an $A$-invariant irreducible constituent and $C$ acts transitively on $\IBr_A(N\mid \chi)$.
	
	Now assume further that $G/N$ is abelian.	Then the following statements hold.
	\item Let $\psi\in\IBr_A(N)$. Then $\Ind_N^G(\psi)$  has an $A$-invariant irreducible constituent and $\IBr_A(G/N)$ acts transitively on $\IBr_A(G\mid \psi)$.
		\item If $G=CN$, then $\IBr_A(G)=\IBr(G\mid \IBr_A(N))$ and $\IBr_A(N)=\IBr(N\mid \IBr_A(G))$.
	  \item Assume that for every $\chi\in\IBr(G)$,
	  $\Res^G_N(\chi)$ is multiplicity-free. Then for $\chi\in\IBr_A(G)$, the set $\IBr_A(CN\mid\chi)$ 
	  is a singleton.
	  If we write $\IBr_A(CN\mid\chi)=\{\Omega(\chi)\}$, then $\chi\mapsto\Omega(\chi)$ gives a bijection between $\IBr_A(G)$ and $\IBr_A(CN)$.
	In addition, $\Omega(\chi\lambda)=\Omega(\chi)\Res^G_{CN}(\lambda)$ for every $\chi\in\IBr_A(G)$ and every $\lambda\in\IBr_A(G/N)$.
  \end{enumerate}
\end{prop}

\begin{proof}
Part	(i) is similar to \cite[Thm.~13.27]{Is76} for Brauer characters.
By Clifford theory, the group $G$ acts transitively on $\IBr(N\mid \chi)$.
Now $A$ acts on both $G$ and $\IBr(N\mid \chi)$.
Note that by Feit--Thompson odd-order theorem, at least one of $A$ and $G$ is solvable.
Then (i) follows by a result of Glauberman on coprime actions (cf.~\cite{Gl64}), for which we also refer the reader to~\cite[Lemma~3.24]{Is08}.

Part (ii) is similar to \cite[Thm.~13.28]{Is76} for Brauer characters.
Since $G/N$ is abelian, as mentioned before, the group $\IBr(G/N)$ acts transitively on $\IBr(G\mid \psi)$. 	
Also $A$ acts on both $\IBr(G/N)$ and $\IBr(G\mid \psi)$ and
thus we can prove (ii) as above.
Part (iii) is similar to Lemma~1.2~(iii) and Lemma~1.3~(iii) of \cite{Un83} for Brauer characters and follows from (i) and (ii) immediately.

For (iv), we note that $\C_{G/N}(A)=CN/N$ and $\C_{G/CN}(A)=1$ by~\cite[Cor.~3.28]{Is08}.
Let $\chi\in\IBr_A(G)$.
Then by (i) there exists $\phi\in\IBr_A(CN\mid\chi)$ and $\psi\in\IBr_A(N\mid\phi)$.
According to the assumption that $\Res^G_N(\chi)$ is multiplicity-free, one knows that $\psi\not\in\IBr(N\mid\phi')$ for any $\phi'\in\IBr(CN\mid\chi)$ with $\phi'\ne\phi$.
On the other hand, since $C$ acts transitively on $\IBr_A(N\mid \chi)$, we see that $\IBr_A(N\mid\chi)=\IBr(N\mid\phi)$.
Also, if $\phi'\in\IBr(CN\mid\chi)$ with $\phi'\ne\phi$, then $\phi'$ is not $A$-invariant.
Thus the set $\IBr_A(CN\mid\chi)$ 
is a singleton and the map $\Omega$ is well-defined and injective.
By (i) and (ii),  $\Omega$  is also surjective and thus (iv) holds.
\end{proof}

By Proposition~\ref{rela-G-CN}~(iv) we have the following which can also be seen as a direct consequence of Glauberman's result.

\begin{cor}\label{rela-abel}
Let $G$ be a finite abelian group and $A$ a finite group acting coprimely on $G$ via automorphisms. 
Let $C=\C_G(A)$.
Then $\Res^G_{C}:\IBr_{A}(G)\to\IBr(C)$ is a bijection.
\end{cor}

The following takes care of Brauer characters of central products (see \cite[\S10.3]{Na18} about the analogue for ordinary characters) and extendibility.

\begin{lem}\label{ext-central-prod} Let the finite group $G=G'  G''$ be a central product of two subgroups $G' ,G''$ over $Z=G'\cap G''\le\ZZ(G)$. Let $\chi'\in\IBr(G')$ and $\chi''\in\IBr(G'')$
with $\IBr(Z\mid \chi' )=\IBr(Z\mid \chi'')$. Then the following holds.
	\begin{enumerate}[\rm(i)]
		\item 	The function $\chi:=\chi' \! .\chi'':G_{\ell'}\to\mathbb{C}$,
		such that 
		$\chi(x)=\chi' (x' )\chi''(x'')$ for $x=x'  x''$ with  $x' \in  G_{\ell'}'$ and $x''\in  G_{\ell'}''$, is an irreducible Brauer character of $G$.
		Conversely, every irreducible Brauer character of $G$ has this form.
		\item Let $D$ be a finite group acting on $G$ via automorphisms such that $G' $ and $G''$ are $D$-stable and $D=D_{\chi'}=D_{\chi''}$.
		If $\chi'$ extends to $G'\rtimes D$ and $\chi''$ extends to $G''\rtimes D$, 
		then $\chi$ extends to $G\rtimes D$.
	\end{enumerate}	
\end{lem}

\begin{proof} The easy proof is left to the reader, noting that it will be applied only to cases where $G''\le\ZZ(G)$.
%
\end{proof}

\begin{prop}\label{central-extension}
Let $A$ act coprimely via automorphisms on $G$, where $A$ and $G$ are finite groups.
Assume that   $G=CZ$, where $C=\C_G(A)$ and $Z\le \ZZ(G)$ is $A$-stable.
Let $D\le\Aut(G)$ such that $Z$ is $D$-stable and $D$ commutes with $A/\C_A(G)$ in $\Aut(G)$.
Then we have the following.
\begin{enumerate}[\rm(i)]
	\item $\Res^{G}_C: \IBr_A(G)\to\IBr(C)$ is bijective and $D$-equivariant. 
	\item Let $\chi\in\IBr_A(G)$ and $\psi=\Res^G_C(\chi)$. Then $\chi$ extends to $G\rtimes D_{\chi}$ if and only if
	$\psi$ extends to $C\rtimes D_\psi$.
		\item $\Res^G_C(\IBr_A(G\mid \nu))=\IBr(C\mid \Res^{Z}_{Z\cap C}(\nu))$ for every $\nu\in\IBr_A(Z)$.
	\item Let $N$ be an $A$-stable normal subgroup of $G$ such that $Z\subseteq  N$ and
$G/N$ is abelian.
Let $M=C\cap N=\C_N(A)$.
Then $\J^G_N(\chi)=\J^C_M(\Res^G_C(\chi)). Z$ (central product) for every $\chi\in\IBr_A(G)$.
\end{enumerate}
\end{prop}

\begin{proof}
	Let $Y= [Z,A]$. It is $A$-stable but also $D$-stable since the actions of $D$ and $A$ on $Z$ commute.
According to a theorem of Fitting (see e.g. \cite[Thm.~4.34]{Is08}), $Z=(C\cap Z)\ti Y$ and
thus $G=C\ti Y$.
By~\cite[Cor.~3.28]{Is08}, $\C_{G/C}(A)=1$ and so $\C_{Y}(A)=1$.
Therefore, the trivial Brauer character $1_{Y_{\ell'}}$ is the unique $A$-invariant irreducible Brauer character of $Y$.
From this $\Res^{G}_C: \IBr_A(G)\to\IBr(C)$  is a bijection with inverse $ \IBr(C)\to\IBr_A(G)$, $\psi\mapsto \psi\ti 1_{Y_{\ell'}}$.  
These two bijections are automatically $D$-equivariant, hence part (i) holds.
Part (ii) follows by Lemma~\ref{ext-central-prod}~(ii).

For $\nu\in\IBr_A(Z)$, one has $\nu=\nu'\ti 1_{Y_{\ell'}}$, where $\nu'\in\IBr_A(C\cap Z)$.
So part (iii) follows from the construction.
Finally, we prove (iv).
Note that $N=CZ\cap N=MZ$ and $C\cap Z=M\cap Z$.
Thus as above, $N=M\ti Y$.
Therefore, we can see that
$\J^G_N(\chi)=\J^C_M(\psi). Y= \J^C_M(\psi). Z$ (central products),
which proves (iv).
\end{proof}

\section{Decomposition matrices, Unitriangularity and Clifford theory}\label{sec-decom-CLI}

Let $H$ be a finite group.
For $\chi\in\Irr(H)$,
we denote the restriction of $\chi$ to $H_{\ell '}$ by $\chi^0$.
The ($\ell$-)decomposition matrix is $(d_{\chi,\phi})_{\chi\in\Irr(H),\phi\in\IBr(H)}$ defined by $$
\chi^0=\sum_{\phi\in\IBr(H)}d_{\chi,\phi}\phi$$ as class function on $H_{\ell'}$ for any $\chi\in\Irr(H)$.
For $\cB\subseteq\Irr(H)$ and $\cC\subseteq\IBr(H)$ we set $$\Dec_{\cB,\cC}:=(d_{\chi,\phi})_{\chi\in\cB,\phi\in\cC}$$ and call it \emph{unitriangular} if it is square lower unitriangular with respect to some orderings of $\cB$ and $\cC$ (necessarily, $|\cB|=|\cC|$).

In this section, we prove the following ``going-up" property.

\begin{thm}\label{TriangUp}
Let $N\unlhd G$ be finite groups with abelian quotient $G/N$. 
 Let $\cB\subseteq \Irr(N)$ and $\cC\subseteq\IBr(N)$, such that both sets are $G$-stable, $|\cB|=|\cC|$ and $\Dec_{\cB,\cC}$ is unitriangular.
Assume that every element of $\cB$ extends to its stabilizer in $G$.

For $\tcC=\IBr(G\mid\cC)$ there is some $\tcB\subseteq \Irr(G\mid\cB)$ with $|\tcB|=|\tcC|$ such that $\Dec_{\tcB,\tcC}$ is unitriangular.
\end{thm}

The choice of orderings in the proof of Theorem~\ref{TriangUp} makes essential use of the following lemma.

\begin{lem}\label{lem_trian}
Let $n\ge1$. Let $u\in\GL_n(\mathbb{C})$ be lower unitriangular and
let $S\le \GL_n(\mathbb{C})$ be a group of permutation matrices commuting with $u$.
	Then there is a permutation matrix $\sigma$ such that $\sigma u\sigma^{-1}$ is lower unitriangular and the orbits of $\sigma S\sigma^{-1}$ on $\{1,2,\ldots,n\}$ are (consecutive) intervals $[a_i,b_i]$.

\end{lem}

\begin{proof} Let $e_1,\ldots,e_n$ be the canonical (ordered) basis of $\mathbb{C}^n$, viewing the latter as the $n$-dimensional space of column vectors. Our claim is that this basis can be reordered so that the orbits of $S$ correspond to intervals in the new basis and $u$ is still lower unitriangular. We use induction on $n$, the case $n=1$ being trivial.

The hypothesis on $u$ is that $u(e_i)\in e_i+\mathbb{C}e_{i+1}+\cdots+\mathbb{C} e_n$ for any $i=1,\ldots,n$.
We have $u(e_n)=e_n$.
Letting $\cO=\{ i_1<i_2<\cdots<i_t=n\}$ be the $S$-orbit of $n$ with respect to its action	on $\{1,2,\ldots,n\}$, we also have $u(e_i)=e_i$ for any $i\in\cO$ since $u$ and $S$ commute.
Let $\cO':=\{1,2,\ldots,n\}\setminus \cO=\{ i'_1<i'_2<\cdots<i'_{t'} \}$.
Let us define the new ordered basis $$(e_{i'_1},e_{i'_2},\ldots, e_{i'_{t'}},e_{i_1},e_{i_2},\ldots,e_{i_t}).$$
Note that $u$ is still lower unitriangular with respect to this basis because $u(e_i)=e_i$ for each $i\in\cO$ while for $i\in\cO'$ the elements of $\{1,2,\ldots,n\}$ after $i$ in the old order are still after $i$ in the new one.
In the new basis $u$ writes as 
$
\left(
\begin{smallmatrix}
u' & 0 \\
* & I_t
\end{smallmatrix}
\right)
$
with lower unitriangular $u'$ and where the last $t$ basis vectors form an orbit under $S$.
The induction hypothesis then gives our claim.
\end{proof}


\begin{lem}\label{lem_orb}
Let $H\le \tH$ be finite groups. If $\tchi\in\Irr(\tH)$ and $\tphi\in\IBr(\tH)$ are such that $d_{\tchi,\tphi}\ne 0$ then for all $\phi\in\IBr(H\mid\tphi)$ there exists $\chi\in\Irr(H\mid \tchi)$ such that $d_{\chi,\phi}\ne 0$.
\end{lem}

\begin{proof}
Recall $\tchi^0=\sum_{\widetilde\psi\in\IBr(\tH)}d_{\tchi,\widetilde\psi}\widetilde\psi$ and $\Res^{\tH}_{H}\tchi=\sum_{\chi\in\Irr(H\mid\tchi)}m_\chi\chi$ for some integers $m_\chi\ge 1$.
Restricting to $H$ we see 
$$\sum_{\chi\in\Irr(H\mid\tchi)}m_\chi\sum_{\psi\in\IBr(H)}d_{\chi,\psi}\psi=\sum_{\widetilde\psi\in\IBr(\tH)}d_{\tchi,\widetilde\psi}\Res^{\tH}_{H}\widetilde\psi.$$
This gives our claim since all decomposition numbers are non-negative integers.
\end{proof}

For the proof of Theorem \ref{TriangUp} we use the so-called extension maps, see for instance \cite[Definition~2.9]{CS13}.
For $N\unlhd G$ and $\cB\subseteq\Irr(N)$, an \emph{extension map with respect to $N\unlhd G$ for $\cB$} is a map $$\Pi:\cB\longrightarrow \coprod_{N\le I\le G} \Irr(I)$$
such that every $\chi\in\cB$ extends to $\Pi(\chi)\in\Irr(G_\chi)$.
If $\cB$ is $G$-stable and every character $\chi\in\cB$ extends to its stabilizer $G_\chi$ in $G$, such a map $\Pi$ exists and can be assumed to be $G$-equivariant by choosing representatives for the right cosets of $G_\chi$ for every $\chi\in\cB$, and a fixed extension of $\chi$.

In connection with Theorem \ref{TriangUp} we now prove the following ``inductive" statement:

\begin{prop}\label{IndTri}
Let $N\lhd M$ be finite groups such that $|M/N|$ is a prime.
Let $\cB\subseteq\Irr(N)$ and $\cC\subseteq\IBr(N)$ both $M$-stable, such that $|\cB|=|\cC|$ and $\Dec_{\cB,\cC}$ is unitriangular. Then:
\begin{enumerate}[\rm(i)]
\item there is a subset $\cB_M\subseteq \Irr(M\mid\cB)$ such that ${|\cB_M|=|\IBr(M\mid\cC)|}$ and $\Dec_{\cB_M,\IBr(M\mid\cC)}$ is unitriangular.
\item If moreover $\ell\nmid |M/N|$ then $\cB_M=\Irr(M\mid\cB)$.
\item Assume that $|M/N|=\ell$ and let $\Lambda$ be an $M$-equivariant extension map with respect to $N\lhd M$ for $\cB$.
Then one can choose $\cB_M=\{ \Ind^{M}_{M_\theta}(\Lambda(\theta))\mid\theta\in\cB \}$.
If moreover $N$ and $M$ are both normal in some group $G$ stabilizing $\cB$ and $\Lambda$ is $G$-equivariant, then the set $\cB_M$ just defined is $G$-stable.
\end{enumerate}
\end{prop}

\begin{proof} First observe that since $M/N$ is cyclic, any element of $\Irr(N)\cup\IBr(N)$ extends to its stabilizer in $M$.
	
By a classical argument on permutation and triangular matrices (see for instance the proof of \cite[7.5]{CS13}) the unitriangularity hypothesis defines a unique bijection
$$f:\cB \stackrel{1-1}{\longrightarrow} \cC$$
making $\Dec_{\cB,\cC}$ unitriangular for some total ordering $\le$ on $\cB$ and $\cC$.	
Since  those sets are $M$-stable, the bijection $f$ is $M$-equivariant by uniqueness and the fact that $d_{{}^\sigma\chi,{}^\sigma\phi}=d_{\chi,\phi}$ for any $\chi\in\Irr(N)$, $\phi\in\IBr(N)$ and $\sigma\in\Aut(N)$ (see also \cite[2.3]{De17}).
According to Lemma \ref{lem_trian} the ordering $\le$ can be chosen so that the $M$-orbits $\cB_1,\cB_2,\ldots,\cB_m$ in $\cB$ satisfy $\chi_i<\chi_j$ for any $\chi_i\in\cB_i$, $\chi_j\in\cB_j$ with $i<j$.
The same is of course true for the $M$-orbits $\cC_i=f(\cB_i)$ on $\cC$.

Let $r:=|M/N|$.

Assume first that $r\ne\ell$.
Define $\tcB_i=\Irr(M\mid\cB_i)$ and $\tcC_i=\IBr(M\mid\cC_i)$, so that Clifford theory implies $\Irr(M\mid\cB)=\coprod_{i=1}^m\tcB_i$ and $\IBr(M\mid\cC)=\coprod_{i=1}^m\tcC_i$
with $|\tcB_i|=|\tcC_i|=r|\cB_i|^{-1}$ for any $i$.

We define a total ordering $\preceq$ on $\Irr(M\mid\cB)$ as follows: We choose an arbitrary total ordering $\preceq$ on each $\tcB_i$ and impose $\tchi_i\prec\tchi_j$ for any $\tchi_i\in\tcB_i$, $\tchi_j\in\tcB_j$ with $i<j$.
We will choose a total ordering $\preceq$ on $\Irr(M\mid\cC)$ but impose already $\tphi_i\prec\tphi_j$ for any $\tphi_i\in\tcC_i$, $\tphi_j\in\tcC_j$ with $i<j$ and will check unitriangularity of the decomposition matrix.

According to Lemma \ref{lem_orb} the unitriangularity of $\Dec_{\cB,\cC}$ implies $d_{\tchi,\tphi}=0$ whenever $\tchi\in\tcB_i$, $\tphi\in\tcC_j$ with $i<j$.

Let us now show that
\begin{equation}\label{equation}
\textrm{for every}\ \tchi\in\tcB_i\ \textrm{there is a unique}\ \tphi\in\tcC_i\ \textrm{such that}\ d_{\tchi,\tphi}\ne 0\ \textrm{and then}\ d_{\tchi,\tphi}=1.
\addtocounter{thm}{1}\tag{\thethm}
\end{equation}

We fix $i$ and $\tchi\in\tcB_i$.
Note first that there is always some $\tphi\in\tcC_i\ \textrm{such that}\ d_{\tchi,\tphi}\ne 0$. Indeed since $\tchi\in\Irr(M\mid\chi)$ for some $\chi\in\cB_i$ and $d_{\chi, f(\chi)}=1$, we can write $\Res_N^M\tchi^0=\sum_{\tphi\in\IBr(M)}d_{\tchi,\tphi}\Res_N^M\tphi \in f(\chi)+{\Bbb N}\IBr(N)$. But for $f(\chi)$ to appear in the first sum, there has to be some $\tphi\in\IBr(M\mid f(\chi))\subseteq\tcC_i$ such that $d_{\tchi,\tphi}\ne 0$. We now proceed to prove the rest of (\ref{equation}).

Assume now $|\tcB_i|=|\tcC_i|=1$ and $|\cB_i|=|\cC_i|=r$. Then $\tchi=\Ind_N^M\chi$ for some (in fact every) $\chi\in\cB_i$ and $\tphi=\Ind_N^M\phi$ for every $\phi\in\cC_i$.
Take $\chi\in\cB_i$ the smallest element for $\le$, then $f(\chi)$ is also $\le$-minimal in $\cC_i$. We have $\tphi=\Ind_N^M f(\chi)$ and the lower unitriangularity of $\Dec_{\cB,\cC}$ via $f$ and $\leq$ implies
$$\tchi^0=\Ind_N^M\chi^0\in \tphi +\sum_{\psi\in\cC,\psi<f(\chi)}\mathbb{C} \Ind_N^M\psi \ +\ \Ind_N^M\Psi$$ where $\Psi\in\mathbb{C}(\IBr(N)\setminus\cC)$.
For every $\chi$, one sees easily that some $\Ind_N^M\psi$ can have the irreducible character $\tphi=\Ind_N^Mf(\chi)$ as a constituent only if $\psi$ is in the $M$-orbit of $f(\chi)$, that is in $\cC_i$.
But the choice made of $\chi\in\cB_i$ implies that $f(\chi)$ is the smallest element of $\cC_i$ for $\le$.
So we indeed get $d_{\tchi,\tphi}=1$ as claimed.

Assume that $|\tcB_i|=|\tcC_i|=r$ and $|\cB_i|=|\cC_i|=1$, so that $\cB_i=\{\theta\}$ and $\tcB_i$ is the set of $r$ different extensions of $\theta$ to $M$.
Let $\tchi\in \tcB_i$ and assume there are $\tphi$ and $\tphi'$ in $\tcC_i$ such that
$$\tchi^0=\tphi+\tphi'+\sum_{\rho\in\IBr(M)}d_\rho\rho$$
for non-negative integers $d_\rho$'s.
Restricting to $N$, we see that $\theta^0$ has the unique element of $\tcC_i=\{f(\theta) \}$ with multiplicity $\ge 2$, which contradicts our assumption that $d_{\theta ,f(\theta)}=1$.
So we get both claims of (\ref{equation}).

From (\ref{equation}) we obtain a bijection $\tcB_i\to\tcC_i$ and via this map a total ordering $\preceq$ on $\tcC_i$ corresponding to the one on $\tcB_i$.
This extends into a total ordering $\preceq$ on $\IBr(M\mid\tcC)$ by setting $\tphi\prec\tphi'$ whenever $\tphi\in\tcC_i$ and $\tphi'\in\tcC_j$ with $i<j$ as said before.
Let us show that $\Dec_{\tcB,\tcC}$ is unitriangular with respect to the given ordering.
By (\ref{equation}) the decomposition matrix is the identity matrix on each $\tcB_i\ti\tcC_i$.
We have seen before that $d_{\tchi,\tphi}=0$ whenever $\tchi\in\tcB_i$ and $\tphi\in\tcC_j$ with $i<j$.
Defining $\cB_M=\Irr(M\mid\cB)$ finishes our proof of (i) and (ii) in the cases where $r\ne\ell$.

It remains to consider the case where $r=\ell$.

Recall that $\Lambda$ denotes an $M$-equivariant extension map with respect to $N\lhd M$ for $\cB$. It always exists since $M/N$ is cyclic.
Let $\cB_i$ be an $M$-orbit in $\cB$, $\theta\in\cB_i$ its $\le$-smallest element and set $$\cB_i':=\{ \Ind^{M}_{M_\theta}(\Lambda(\theta)) \}$$ and $\tcC_i=\IBr(M\mid\cC_i)$. The latter is a singleton since $\cC_i$ is an $M$-orbit and, as said at the beginning of Section~\ref{Preliminiaries}, $\IBr(M\mid\phi)$ is an $\IBr(M/N)$-orbit for any $\phi\in\IBr(N)$.
Note that in (iii) the set $\cB_M:=\cup_i\cB_i'=\{ \Ind^{M}_{M_\theta}(\Lambda(\theta))\mid\theta\in\cB \}$ is $G$-stable whenever $\Lambda$ is $G$-equivariant.
Again we define an ordering $\preceq$ on $\cB_M$ and $\tcC:=\cup_i\tcC_i$ by $\chi_i\prec\chi_j$, whenever $\chi_i\in\cB_i'$ and $\chi_j\in\cB_j'$ for $i<j$, and we set analogously $\phi_i\prec\phi_j$, whenever $\phi_i\in\tcC_i$ and $\phi_j\in\tcC_j$ for $i<j$.
Since the sets $\cB_i'$ and $\tcC_i$ are singletons, this gives a total ordering on $\cB_M$ and $\tcC$.
The corresponding decomposition matrix is lower triangular by that choice thanks to Lemma \ref{lem_orb}.
With the arguments above proving (\ref{equation}) we see easily that $d_{\tchi,\tphi}=1$ whenever $\tchi\in\cB_i'$ and $\tphi\in\tcC_i$. 
This finishes the proof of (i) and (iii), the last statement of (iii) being clear from the rest.
\end{proof}

Now we are ready to prove our Theorem \ref{TriangUp}.

\begin{proof}[Proof of Theorem \ref{TriangUp}]
Let $\Lambda$ be a $G$-equivariant extension map with respect to $N\unlhd G$ for $\cB$.
This map exists by our assumption.

We use induction on $|G/N|$.
The case $G=N$ is trivial.
Otherwise, since $G/N$ is abelian, there is $N\le M\le G$ with prime index $|M/N|$ and we apply Proposition \ref{IndTri}.
Now let us ensure that for $M\le G$ all the hypotheses of Theorem \ref{TriangUp} are satisfied with respect to $\cB_M\subseteq\Irr(M)$, and $\IBr(M\mid\cC)\subseteq\IBr(M)$.
According to Proposition \ref{IndTri} it remains only to prove that every $\chi\in\cB_M$ extends to $G_\chi$.
Since $\cB_M\subseteq\Irr(M\mid\cB)$ there is $\theta\in\cB\cap\Irr(N\mid\chi)$.
Clifford theory implies that $$\chi=\Ind^M_{M\cap G_\theta}(\Res^{G_\theta}_{M\cap G_\theta}(\Lambda(\theta))\la)$$ where $\Lambda$ is the extension map introduced before and $\la\in\Irr(M\cap G_\theta\mid 1_N)$, a linear character.
Let's note that $G_\chi=G_\theta M$. Indeed $G_\theta\leq G_\chi$ by the above formula for $\chi$, while the inclusion $G_\chi\leq G_\theta M$ results from $\Res^M_N\chi =\sum_{m\in M/M_\theta}{}^m\theta$.
Since $G/N$ is abelian, $\la$ extends to some $\widetilde\la\in\Irr(G_\theta\mid 1_N)$.
Then $\Ind_{G_\theta}^{G_\chi}(\Lambda(\theta)\widetilde\la)$ is an extension of $\chi$ as can be seen from the Mackey formula.
This proves that the assumptions of Theorem \ref{TriangUp} are satisfied by $M\le G$ and the sets $\cB_M\subseteq\Irr(M)$ and $\IBr(M\mid\cC)$.
This obviously gives our claim by induction.
\end{proof}

The proof of Theorem \ref{TriangUp} allows for a generalization, since when $G/N$ is an $\ell'$-group the proof does not use that $G/N$ is abelian or the existence of an extension map. Recall that we consider here $\ell$-Brauer characters.

\begin{cor}\label{cor}
Let $N\unlhd G$, $\cB\subseteq \Irr(N)$ and $\cC\subseteq\IBr(N)$ such that $|\cB|=|\cC|$ and  $\Dec_{\cB,\cC}$ is unitriangular. 
If $\cB$ and $\cC$ are both $G$-stable and $G/N$ is a solvable $\ell'$-group, then $\tcB:=\Irr(G\mid\cB)$ and $\tcC:=\IBr(G\mid\cC)$ have same cardinality and $\Dec_{\tcB,\tcC}$ is unitriangular.
\end{cor}

\begin{proof} We use induction on $|G/N|$, the case $N=G$ being trivial. Assume now
 $G/N$ is a solvable non-trivial $\ell'$-group, so there exists $N\leq K\unlhd G$ with $G/K$ of prime order $\not=\ell$. Induction tells us that $\cB ':=\Irr(K\mid\cB)$ and $\cC ':=\IBr(K\mid\cC)$ have same cardinality and $\Dec_{\cB ',\cC '}$ is unitriangular. Note also that from their definition $\cB '$ and $\cC '$ are clearly $G$-stable. We may then apply Proposition \ref{IndTri} (i) and (ii) to the inclusion $K\lhd G$ and the sets $\cB '$ and $\cC '$. We get that $\Irr (G\mid \cB ')$ and $\IBr (G\mid \cC ')$ have same cardinality and that the associated decomposition matrix is unitriangular. This gives our claim since clearly $\Irr (G\mid \cB ') = \Irr (G\mid \Irr(K,\cB))=\Irr (G\mid \cB)$ and $\IBr (G\mid \cC ') = \IBr (G\mid \IBr(K,\cC))=\IBr (G\mid \cC)$.	
\end{proof}

\begin{cor}\label{going-up-special}
	Let $N\unlhd G$ be finite groups with solvable quotient $G/N$ and $Z\le\ZZ(G)$ such that 
	 $\ell\nmid |G/NZ|$ and $\ell\nmid |Z\cap N|$. 
	Let $\cB\subseteq \Irr(N)$ and $\cC\subseteq\IBr(N)$ such that $|\cB|=|\cC|$ and  $\Dec_{\cB,\cC}$ is unitriangular. 
	Assume that $\cB$ and $\cC$ are both $G$-stable, and let 
	$\tcB=\Irr(G\mid\cB)\cap\Irr(G\mid 1_{Z_\ell})$ and
	$\tcC=\IBr(G\mid\cC)$. 
	Then $\Dec_{\tcB,\tcC}$ is unitriangular.
\end{cor}

\begin{proof}
By the assumption, $G_1:=N Z_\ell=N\ti Z_\ell$ is a direct product.
	So $\cC$ can be also regarded as a subset of $\IBr(G_1)$.
Let $\cB_1=\{\chi\times 1_{Z_\ell}\mid\chi\in\cB\}$.
Then  $\cB_1$ is $G$-stable and $\Dec_{\cB_1,\cC}=\Dec_{\cB,\cC}$ is unitriangular.
By Corollary \ref{cor},
 $\Dec_{\tcB,\tcC}$ is unitriangular for $\tcB=\Irr(G\mid\cB_1)$ and $\tcC=\IBr(G\mid\cC)$ since $G/G_1$ is a solvable $\ell'$-group.
Here, $\tcB=\Irr(G\mid\cB)\cap\Irr(G\mid 1_{Z_\ell})$, and this completes the proof.
\end{proof}

\begin{defn}\label{DefBasic}
If $\ell$ is a prime, $H$ is a finite group and $\mathcal I\subseteq \Irr(H)$, a subset $\cB\subseteq \mathcal I$ is called \emph{basic} if $(\chi^0)_{\chi\in\cB}$ is a basis for the $\Z$-lattice generated by $\{\chi^0\mid \chi\in\mathcal I\}$.
If $B$ is a union of $\ell$-blocks of $H$, a \emph{basic set} for $B$ is a basic subset of $\Irr(B)$.
Such a basic set $\cB$ is called \emph{unitriangular} whenever the decomposition matrix $\Dec_{\cB,\IBr(B)}$ is unitriangular for some ordering of the rows and columns. Conversely, note that if $\cB$ is a subset of $\Irr(B)$ such that $\Dec_{\cB,\IBr(B)}$ is square unitriangular then $\cB$ is a basic set for $B$.
\end{defn}

Here are some comments on Theorem  \ref{TriangUp} and Proposition \ref{IndTri}.

\begin{rmk}\label{remTri}
\begin{enumerate}[\rm(i)]
	\item Theorem \ref{TriangUp} should be compared with Denoncin's ``going-down" statement \cite[Thm.~2.5]{De17}.
For $N\unlhd G$ with cyclic quotient $G/N$ a unitriangular $\IBr(G/N)$-stable basic set $\tcB\subseteq\Irr(G)$ allows to construct one for $N$ if the underlying bijection $\widetilde f:\tcB\to\IBr(G)$ satisfies $G_\chi=G_\phi$ whenever $\chi\in\Irr(N\mid\tchi)$, $\phi\in\IBr(N\mid\widetilde f(\tchi))$ for some $\tchi\in\tcB$.
	
	In the situation of Theorem \ref{TriangUp} the given unitriangular basic set $\cB$ is $G$-stable and hence $G_\chi=G_{f(\chi)}$ for every $\chi\in\cB$, where $f:\cB\to\cC$ is the bijection defined by the unitriangular shape of the decomposition matrix.	
	\item	In the situation of Theorem \ref{TriangUp}, if $G/N$ is an $\ell$-group, then one can take $\tcB=\{ \Ind_{G_\theta}^G(\Lambda(\theta))\mid\theta\in\cB \}$, where $\Lambda$ is the $G$-equivariant extension map with respect to $N\unlhd G$ for $\cB$.
	The set $\tcB$ is obtained by iterating the construction in Proposition \ref{IndTri} (iii).
	Let $N\le M\le G$ with prime index $|M/N|$.
	We use as an extension map $\Lambda_1$ with respect to $N\unlhd M$ for $\cB$ the map given by $\theta\mapsto \Res^{G_\theta}_{M_\theta}(\Lambda(\theta))$.
	Then $\cB_M=\{\Ind^M_{M_\theta}(\Lambda_1(\theta))\mid\theta\in\cB \}$.
	Assuming that $M\unlhd G$, which is possible since $G/N$ is an $\ell$-group, one may use as a $G$-equivariant extension map $\Lambda_2$ with respect to $M\unlhd G$ for $\cB_M$ the map given by $\Ind^M_{M_\theta}(\Lambda_1(\theta))\mapsto\Ind^{G_\theta M}_{G_\theta}(\Lambda(\theta))$ for $\theta\in\cB$.
	Note that $\Lambda_2$ is a well-defined extension map since $\Lambda$ is $G$-equivariant.
	Now by composing the two constructions we get $\tcB=\{\Ind^G_{G_\phi}(\Lambda_2(\phi))\mid\phi\in\cB_M \}$ and this coincides with $\{\Ind^G_{G_\theta}(\Lambda(\theta))\mid\theta\in\cB\}$ by the definition of $\Lambda_1$ and $\Lambda_2$.	
	\item	At least when the quotient $G/N$ is an $\ell$-group the assumption in Theorem~\ref{TriangUp} that every element of $\cB$ extends to its stabilizer in $G$ seems indeed necessary. Let for instance $G$ be a non-abelian $\ell$-group and $N$ its derived subgroup, so that $[N,G]\lneq N\lneq G$. Then $N$ has a non-trivial linear character $\chi\in\Irr(N)$ with $[N,G]$ in its kernel. Now $\chi$ is $G$-invariant and we take $\cB := \{\chi  \}$, while $\cC$ and $\tcC$ are forced to be reduced to the trivial Brauer character. One has $d_{\chi ,1_N^0}=\chi(1)=1$. But for every $\tchi $ of $\Irr(G\mid \chi)$ one has $\tchi(1)\not= 1$ since otherwise $N$ would be in the kernel of $\tchi$ and therefore $\chi =1_N$. Then $d_{\tchi ,1_G^0} =\tchi(1)\not= 1$ which makes impossible to find $\tcB\subseteq \Irr(G\mid \cB)$ such that $\Dec_{\tcB,\tcC}$ is unitriangular.

	\end{enumerate}
\end{rmk}

\section{The inductive Brauer--Glauberman condition}
\label{sec-iBGC}

In this section, we recall the (iBG) condition from \cite{SV16} (see Definition~\ref{iBGC}) and give a new criterion to check it.

Let $\mathbf R$ be the ring of algebraic integers in $\mathbb{C}$ and $I$ be a maximal ideal of $\mathbf R$ with $\ell\mathbf R \subseteq I$. Then the quotient $\mathbf F=\mathbf R/I$ is an algebraically closed field of characteristic $\ell$, see \cite[\S 2]{Na98}.
Let ${}^*:\mathbf R\to  \mathbf F$ be the natural ring homomorphism.

\subsection{Modular character triples}

The notion of character triples and isomorphisms between them is well-known and has important applications. A more restrictive relation, the \emph{central isomorphisms} of character triples,
which was first introduced in~\cite{NS14}, 
is useful in the context of reduction theorems.
Here we refer to the exposition
given in~\cite{Na18} and~\cite{Sp18}.
We will recall a similar notion for modular representations, 
the \emph{central isomorphisms of modular character triples}, which
were used in~\cite{SV16} for constructing the (iBG) condition. 

Let us start by recalling some facts on modular character triples from~\cite[\S8]{Na98}.
If $N\unlhd G$  and $\theta\in\IBr(N)$ is $G$-invariant, then the triple $(G,N,\theta)$ is called a \emph{modular character triple}.
Let $\mathcal D$ be an $\mathbf F$-representation affording the Brauer character $\theta$.
According to~\cite[Thm.~8.14]{Na98}, there is a projective $\mathbf F$-representation $\mathcal P$ of $G$ such that $\mathcal P|_N=\mathcal D$. 
Moreover, $\mathcal P$ can be chosen such that its factor set $\alpha$ satisfies $\alpha(g,n)=\alpha(n,g)=1$ for every $g\in G$ and $n\in N$. In this situation, we say that \emph{$\mathcal P$ is a projective representation of $G$ associated to $\theta$}.
It follows that $\alpha$ can be seen as a map on $G/N\times G/N$.
Furthermore, if $c\in \C_G(N)$, then $\mathcal P(c)$ is a scalar matrix by Schur's Lemma.

Now we recall the definition of a central isomorphism between modular character triples from \cite[Def.~3.3]{SV16} or \cite[Def.~4.19]{Sp18}.

\begin{defn}\label{defn-centiso}
Let $(G,N,\theta)$ and $(H,M,\varphi)$ be modular character triples satisfying the following conditions.
\begin{enumerate}[(i)]
\item $G=NH$, $M=N\cap H$ and $\C_G(N)\le H$.
\item There exist a projective representation $\mathcal P$ of $G$ associated to $\theta$ with factor set $\alpha$ and  a projective representation $\mathcal P'$ of $H$ associated to $\varphi$ with factor set $\alpha'$ such that
\begin{itemize}
\item[(ii1)] $\alpha'=\alpha|_{H\times H}$, and
\item[(ii2)] for every $c\in\C_G(N)$ the scalar matrices $\mathcal P(c)$ and $\mathcal P'(c)$ are associated with the same scalar.
\end{itemize}
\end{enumerate}
Then we say the modular character triples $(G,N,\theta)$ and $(H,M,\varphi)$ are \emph{central isomorphic} and write $(G,N,\theta)\succ_{c}(H,M,\varphi)$.
\end{defn}

\subsection{On central isomorphisms}
We will make use of the following results on central isomorphisms between modular character triples.

Let $L$ be a finite group.
For $\psi\in\IBr(L)$, we let $\mathcal D$ be an $\mathbf F$-linear representation of $L$ affording $\psi$.
If we define $\psi^*(g)=\psi(g_{\ell'})^*$ for every $g\in L$, then by~\cite[Lemma~2.4]{Na98}, $\psi^*$ is the trace function of $\mathcal D$.

The following lemma is similar to~\cite[Lemma~2.11]{Sp12} for the case of Brauer characters. See \cite[2.7]{FR21} for the proof.

\begin{lem}\label{twexten}
Let $L\le \widetilde L \unlhd X$ be finite groups with $L\unlhd X$ and abelian quotient $\widetilde L/L$,
$\psi\in\IBr(L)$ with $X_\psi=X$ and $\widetilde \psi \in\IBr(\widetilde L)$ an extension of $\psi$ to $\widetilde L$.
Assume that there exists a group $C\le X$ with $\widetilde L\cap C=L$ and $X=\widetilde LC$,
such that $\psi$ extends to $C$.
Then there exists a projective representation $\mathcal P$ associated to the character triple $(X,L,\psi)$ such that the
factor set $\alpha$ of $\mathcal P$ satisfies $\alpha(lcL,l'c'L)=\lambda_c^*(l')$, for $c,c'\in C$, and $l,l'\in \widetilde L$,
where for every $c\in C$, the linear Brauer character $\lambda_c\in \IBr(\widetilde L/L)$ is determined by
$\widetilde\psi=\lambda_c \widetilde\psi^c$.
\end{lem}

We use the above through the following consequence.

\begin{prop}\label{twostepext}
Let $A$ be a finite group and $B\le A$, $N\unlhd A$, $M=N\cap B$ such that $A=NB$.
Let $\theta\in \IBr(N)$ and $\varphi\in\IBr(M)$.
Assume that we have the following conditions.
\begin{enumerate}[\rm(i)]
\item $X$ is a normal subgroup of $A$ such that $N\le X$ and $X/N$ is abelian. $Y$ is a normal subgroup of $B$ such that $M\le Y\le X\cap B$.
       Assume further $\C_A(N)\subseteq Y$.
\item There exist subgroups $H\le A$ and $L\le B\cap H$ satisfying $N=X\cap H$ and  $M=Y\cap L$.
\item $A_\theta=X_\theta H_\theta$, $B_\varphi=Y_\varphi L_\varphi$ and $B_\varphi=B_\theta$.
\item $\theta$ extends to $X_\theta$ and $H_\theta$ while $\varphi$ extends to $Y_\varphi$ and $L_\varphi$.
\item There exist characters $\chi\in\IBr(X\mid\theta)$ and $\psi\in\IBr(Y\mid\varphi)$ satisfying
     \begin{itemize}
\item $\IBr(\C_A(N)\mid\chi)=\IBr(\C_A(N)\mid\psi)$, and
\item for any $l\in L_\varphi$, if  $\chi=\lambda\chi^l$ for some $\lambda\in\IBr(X/N)$,
     then $\psi=\Res^X_Y(\lambda) \psi^l$.
\end{itemize}
\end{enumerate}
Then $(A_\theta,N,\theta)\succ_{c} (B_\varphi,M,\varphi)$.
\end{prop}

\centerline
{$\xymatrix{&&&A\ar@{-}[rdd]\ar@{-}[llldd]\ar@{-}[ddddd]&\\&&&&
		\\H\ar@{-}[ddddd]&&&&X\ar@{-}[ddddd]
		\\&&&&
		\\&&N\ar@{-}[rruu]\ar@{-}[lluu]\ar@{-}[ddddd]&&
		\\&&&B\ar@{-}[rdd]\ar@{-}[llldd]&\\&&&&
		\\H\cap B&&&&X\cap B
		\\&L\ar@{-}[lu]&&Y\ar@{-}[ru]&
		\\&&M\ar@{-}[ru]\ar@{-}[lu]&&
	\\&&&\mathrm{C}_A(N)\ar@{-}[uu]&}$}

\begin{proof}
First we have $A_\theta=B_\varphi N$ and $M=B_\varphi \cap N$ since $B_\varphi=B_\theta$.
Also 
 $Y_\varphi=Y_\theta$,  $L_\varphi=L_\theta$ and
$\C_{A_\theta}(N)\le B_\varphi$ and then condition (i) of Definition of~\ref{defn-centiso} holds.
Let $\widetilde\theta$ be the extension of $\theta$ to $X_\theta$ with $\Ind^X_{X_\theta}(\widetilde\theta)=\chi$ and
$\widetilde\varphi$ be the extension of $\varphi$ to $Y_\varphi$ with $\Ind^Y_{Y_\varphi}(\widetilde\varphi)=\psi$.
Then $$\IBr(\C_A(N)\mid\widetilde\theta)=\IBr(\C_A(N)\mid\chi)=\IBr(\C_A(N)\mid\psi)=\IBr(\C_A(N)\mid\widetilde\varphi)=\{ \epsilon\}$$
for some linear Brauer character $\epsilon\in\IBr(\C_A(N))$.
So we only need to prove the condition (ii1) of Definition~\ref{defn-centiso} now.

By Lemma~\ref{twexten} we have a projective representation $\mathcal P$ associated to $(A_\theta,N,\theta)$, such that the factor set $\alpha$ of $\mathcal P$ satisfies
(a) $\mathcal P(z)=\epsilon^*(z)I_{\theta(1)}$ for $z\in \C_A(N)$,
(b) $\alpha(x_1x_2,x_1'x_2')=\mu_{x_2}^*(x_1')$, for $x_1,x_1'\in X_\theta$, and $x_2,x_2'\in H_\theta$,
where for every $x_2\in H_\theta$, the linear Brauer character $\mu_{x_2}\in \IBr( X_\theta/N)$ is determined by
$\widetilde\theta=\mu_{x_2} \widetilde\theta^{x_2}$.

Analogously,
we have a projective representation $\mathcal P'$ associated to $(B_\varphi,M,\varphi)$,
such that the factor set $\alpha'$ of $\mathcal P'$ satisfies
(a) $\mathcal P'(z)=\epsilon^*(z)I_{\varphi(1)}$ for $z\in \C_A(N)$,
(b) $\alpha'(x_1x_2,x_1'x_2')=\mu_{x_2}'^*(x_1')$, for $x_1,x_1'\in Y_\varphi$, and $x_2,x_2'\in L_\varphi$,
where for every $x_2\in L_\varphi$, the linear Brauer character $\mu_{x_2}'\in \IBr( Y_\varphi/M)$ is determined by
$\widetilde\varphi=\mu_{x_2}' \widetilde\varphi^{x_2}$.

For some $l\in L_\varphi$, and $\lambda\in\IBr(X_\theta/N)$, we have $\widetilde\theta=\lambda\widetilde\theta^l$.
Then $\chi=\widetilde\lambda\chi^l$,
where $\widetilde\lambda\in\IBr(X/N)$ is some extension of $\lambda$.
By the hypothesis, $\psi=\Res^{X}_{Y}(\widetilde\lambda)\psi^l$.
So $\widetilde\varphi=
\Res^{X}_{Y_\varphi}(\widetilde\lambda)\widetilde\varphi^l$.
It is obvious that $\Res^{X}_{Y_\varphi}(\widetilde\lambda)=\Res^{X_\theta}_{Y_\varphi}(\lambda)$.
Hence $\widetilde\varphi=\Res^{X_\theta}_{Y_\varphi}(\lambda)\widetilde\varphi^l$, and
then $\Res^{X_\theta}_{Y_\varphi}(\mu_{x_2})=\mu_{x_2}'$. That is $\alpha|_{B_\varphi\times B_\varphi}=\alpha'$,
which completes the proof.
\end{proof}

\subsection{The inductive Brauer--Glauberman condition}

Using the definition above, the inductive Brauer--Glauberman condition from~\cite[Def.~6.1]{SV16} can be written in the following way. For the definition of \textit{ fake $m$-th Galois action},
see~\cite[\S 4]{SV16}.

\begin{defn}\label{iBGC}
	Let $S$ be a non-abelian finite simple group, $G$ the universal covering group of $S$ and $\ell$ a prime number dividing $|S|$.
	We say that $S$ satisfies the \emph{inductive Brauer--Glauberman (iBG) condition} for $\ell$~(or we say the \emph{inductive Brauer--Glauberman  condition} holds for $S$ and $\ell$)~ if for every $B\le \Aut(G)$, with $\gcd(|G|,|B|)=1$ the following conditions are satisfied:
	\begin{enumerate}[(I)]
		\item For $Z:=\ZZ(G)$, $\Gamma:=\C_{\Aut(G)}(B)$ and $C:=\C_G(B)$ there exists a $\Gamma$-equivariant bijection $$\Omega_B: \ \IBr_{B}(G)\to \IBr_{B}(CZ),$$ such that for every $\theta\in\IBr_{B}(G)$,
		\begin{equation}\label{equaiton-22}
			(G\rtimes\Gamma_\theta,G,\theta)\succ_{c} (CZ\rtimes\Gamma_\theta,CZ,\Omega_B(\theta)).
			\addtocounter{thm}{1}\tag{\thethm}
		\end{equation}
		\item For every non-negative integer $m$ with $(|G|,m)=1$, there exists a fake $m$-th Galois action on $\IBr(C)$ with respect to $C\rtimes \Gamma$.
	\end{enumerate}
\end{defn}

We say \emph{the (iBG) condition
	holds for simple group $S$} if the (iBG) condition holds for $S$ and any prime $\ell$ dividing $|S|$.

Due to the following theorem by the second author and Vallejo  there is considerable interest in
verifying the (iBG) condition for all finite non-abelian simple groups:

\begin{thm}[{\cite[Thm.~A]{SV16}}]\label{main-thm-SV16}
	Let $G$, $A $ be finite groups and let $\ell$ be a prime. Suppose that $A$ acts coprimely on $G$ via automorphisms.
	Suppose that all finite non-abelian simple groups involved in G satisfy the (iBG) condition for $\ell$.
	Then  Conjecture~\ref{conj-corr} is true for $G$, $A$ and the prime $\ell$.
\end{thm}

By~\cite[Rmk.~6.2]{SV16}, for any simple group $S$, the  (iBG) condition in Definition~\ref{iBGC} is equivalent to its weaker version, in which $B$ is taken to be the trivial group in condition (II) of  Definition~\ref{iBGC}.
Recenly,  Farrell and Ruhstorfer proved the condition (II) of the  (iBG) condition and we state this as follows. 

\begin{thm}[{\cite[Thm.~A]{FR21}}]\label{fake-g-act}
	Let $S$ be a non-abelian simple group and let $G$ be the universal covering group of $S$. Then
	for all non-negative integers m such that  $\gcd(|G|,m)=1$, there exists a fake $m$-th Galois action on $\IBr(G)$
	with respect to $G\unlhd  G \rtimes \Aut(G)$.
\end{thm}

From this, 
when considering the  (iBG) condition,
we only need to establish the Definition~\ref{iBGC} (I), a central isomorphism between modular character triples.

\begin{cor}\label{exceptional-covering}
	Let $S$ be a non-abelian simple group and let $G$ be the universal covering group of $S$. 
	If $G$ is an exceptional covering group of a simple group of Lie type, then $S$ satisfies the (iBG) condition.
\end{cor}

\begin{proof}
	By the assumption, $S$ is one of the simple groups of Lie type in~\cite[Table~6.1.3]{GLS98}.
	Case by case calculations show that
	the set of primes dividing $|G|$ and $|\Aut(G)|$ are the same.
	Then the group $B$ in Definition~\ref{iBGC} should be trivial and thus condition (I) holds.
	Therefore, $S$ satisfies the (iBG) condition by Theorem~\ref{fake-g-act}.
\end{proof}

\subsection{A criterion for the inductive Brauer--Glauberman condition}
\label{sec:cri-iBGC}

We will state a technical result which may be useful for the verification of  the  (iBG)  condition for simple groups of Lie type.
It is a generalisation of~\cite[Prop.~5.9]{NST17}.

\begin{thm}\label{criforlie}
Let $S$ be a finite non-abelian simple group, $G$ the universal covering group of $S$ and $\ell$ a prime number dividing $|S|$.
Suppose that there exists a finite group $\widetilde G$ with $G\unlhd \widetilde G$ and a finite group $D$ such that $\widetilde G\rtimes D$ is defined, $G$ is $D$-stable, $\widetilde G/G$ is abelian, $\C_{\widetilde G\rtimes D}(G)=\ZZ(\widetilde G)$,  and $\widetilde G\rtimes D$ induces on $G$ all automorphisms of $G$.
Assume that for every $B\le D$ with $\gcd(|G|,|B|)=1$,
and $C:=\C_G(B)$, $\widetilde C:=\C_{\widetilde G}(B)$, the following conditions are satisfied.
\begin{enumerate}[\rm(i)]
\item 
   \begin{enumerate}[\rm(a)]
   	\item $B\le \ZZ(D)$, $\gcd(|\tG|,|B|)=1$,
    \item every $\chi\in\IBr(G)$ extends to its stabiliser $\widetilde G_\chi$, and
   \item every $\psi\in\IBr(C)$ extends to its stabiliser $\widetilde C_\psi$.
\end{enumerate}
\item      For every $\widetilde\chi\in\IBr_B(\widetilde G)$ there exists some $\chi_0\in \IBr(G\mid\widetilde\chi)$ such that
    \begin{enumerate}[\rm(a)]    	
     \item $(\widetilde G\rtimes D)_{\chi_0}=\widetilde G_{\chi_0}\rtimes D_{\chi_0}$, and
    \item $\chi_0$ extends to $G\rtimes D_{\chi_0}$.    
    \end{enumerate}
\item  For every $\widetilde\psi\in\IBr_B(\widetilde C)$ there exists some $\psi_0\in \IBr(C\mid\widetilde\psi)$ such that
    \begin{enumerate}[\rm(a)]
    \item $(\widetilde C\rtimes D)_{\psi_0}=\widetilde C_{\psi_0}\rtimes D_{\psi_0}$, and
    \item $\psi_0$ extends to $C\rtimes D_{\psi_0}$.
      \end{enumerate}
\item  There exists a $D$-equivariant bijection $\widetilde\Omega_B: \IBr_B(\widetilde G)\to \IBr(\widetilde C)$ between $\IBr_B(\widetilde G)$ and $\IBr(\widetilde C)$ with
    \begin{enumerate}[\rm(a)]
    \item $\tC G\cap \J^{\tG}_{G}(\widetilde\chi)=G. \J^{\tC}_{C}(\widetilde{\Omega}_B(\widetilde\chi))$ (central product)
 for every $\widetilde\chi\in\IBr_B(\widetilde G)$,	
    \item $\widetilde\Omega_B( \IBr_B(\widetilde G\mid\nu))=\IBr_B(\widetilde C\mid\nu)$ for every $\nu\in\IBr(\ZZ(\widetilde G)\cap \tC)$, and
    \item $\widetilde\Omega_B(\widetilde\chi\lambda)=\widetilde\Omega_B(\widetilde\chi)\Res^{\tG}_{\tC}(\lambda)$ for every $\widetilde\chi\in\IBr_B(\widetilde G)$ and $\lambda\in\IBr_B(\widetilde G/G)$.
    \end{enumerate}
\end{enumerate}
Then the  (iBG)  condition holds for $S$ and $\ell$.
\end{thm}

\begin{proof}
Thanks to Theorem~\ref{fake-g-act}, it suffices to show condition (I) of Definition~\ref{iBGC}.	
First by~\cite[Cor.~3.28]{Is08},	
we have $\C_{\Aut(G)}(B)=\C_{\tG\rtimes D/\ZZ(\tG)}(B)
= (\tC \rtimes D)\ZZ(\tG)/\ZZ(\tG)$, which implies that $\tC \rtimes D$ induces on $G$ all $\C_{\Aut(G)}(B)$.		
Let $Z:=\ZZ(G)$.

\centerline
{$\xymatrix{&&&\tC G\rtimes D\ar@{-}[rdd]\ar@{-}[llldd]\ar@{-}[ddddd]&\\&&&&
		\\G\rtimes D\ar@{-}[ddddd]&&&&\tC G\ar@{-}[ddddd]
		\\&&&&
		\\&&G\ar@{-}[rruu]\ar@{-}[lluu]\ar@{-}[ddddd]&&
		\\&&&\tC Z\rtimes D\ar@{-}[rdd]\ar@{-}[llldd]&\\&&&&
		\\CZ\rtimes D&&&&\tC Z
		\\&&&&
		\\&&CZ\ar@{-}[rruu]\ar@{-}[lluu]&&}$}
	
	\medskip
	
By~\cite[Thm.~3.7]{SV16} or~\cite[Thm.~3.5]{Sp17},
it now suffices to prove that for every $B\le D$ with $(|G|,|B|)=1$,
there exists a $(\tC\rtimes D)$-equivariant bijection $\Omega_B: \ \IBr_{B}(G)\to \IBr_{B}(CZ),$ such that
\begin{equation}\label{equa-1}
((\tC G\rtimes D)_\chi,G,\chi)\succ_{c} ((\tC Z\rtimes D)_\chi,CZ,\Omega_B(\chi))
\addtocounter{thm}{1}\tag{\thethm}
\end{equation}
for every $\chi\in\IBr_{B}(G)$.

By Corollary~\ref{rela-abel}, we know that
the restriction  $$\Res^{(\ZZ(\tG)\cap \tC)Z}_{\ZZ(\tG)\cap \tC}:\IBr_B((\ZZ(\tG)\cap \tC)Z) \to \IBr(\ZZ(\tG)\cap \tC)$$ is bijective.
According to Proposition~\ref{central-extension}, we also have a $D$-equivariant bijection
$$\Res^{\tC Z}_{\tC}:\IBr_B(\tC Z)\to \IBr(\tC).$$
Furthermore,
\begin{itemize}
	\item  $\J^{\tC Z}_{CZ}(\widetilde\psi)= \J^{\tC}_C(\Res^{\tC Z}_{\tC}(\widetilde\psi)). Z$ (central product)  for every $\widetilde\psi\in\IBr_B(\widetilde C Z)$, and
	\item $\Res^{\tC Z}_{\tC}(\IBr_B(\tC Z\mid \nu))=\IBr(\tC\mid\Res^{(\ZZ(\tG)\cap \tC)Z}_{\ZZ(\tG)\cap \tC}(\nu))$ for every $\nu\in\IBr_B((\ZZ(\tG)\cap \tC)Z)$.
\end{itemize}
Using again Proposition~\ref{central-extension}, we know from (iii) that
for every $\widetilde\psi\in\IBr_B(\widetilde CZ)$ there exists some $\psi_0\in \IBr(CZ\mid\psi)$ such that
 $(\widetilde CZ\rtimes D)_{\psi_0}=(\widetilde CZ)_{\psi_0}\rtimes D_{\psi_0}$ and
 $\psi_0$ extends to $CZ\rtimes D_{\psi_0}$.

By Proposition~\ref{rela-G-CN} (iv), we have a bijection $\widetilde\Omega''_B$ between $\IBr_B(\tG)$ and $\IBr_B(\tC G)$ which is automatically $D$-equivariant. 
Also from the construction of $\widetilde\Omega''_B$, we have 
\begin{itemize}
\item $\widetilde\Omega''_B( \IBr_B(\widetilde G\mid\nu))=\IBr_B(\widetilde C G\mid\nu)$ for every $\nu\in\IBr(\ZZ(\widetilde C G))$, and 
\item  $\widetilde\Omega''_B(\widetilde\chi\lambda)=\widetilde\Omega''_B(\widetilde\chi)\Res^{\tG}_{\tC G}(\lambda)$ for every $\widetilde\chi\in\IBr_B(\widetilde G)$, and $\lambda\in\IBr_B(\widetilde G/G)$.
\end{itemize}
In addition, it is easy to check that
$\J^{\tC G}_G(\widetilde\Omega''_B(\widetilde\chi))=\tC G\cap \J^{\tG}_G(\widetilde\chi)$  for every $\widetilde\chi\in\IBr_B(\widetilde G)$.

By~\cite[Cor.~3.28]{Is08},  $\C_{\tG/G}(B)=\tC G/G$.
Then applying Corollary~\ref{rela-abel}, we know that $$\Res^{\tG/G}_{\tC G/G}:\IBr_B(\tG/G)\to \IBr(\tC G/G)$$ is bijective.
By (iv) we obtain a $D$-equivariant bijection $\widetilde\Omega'_B: \ \IBr_B(\widetilde CG)\to \IBr_B(\widetilde CZ)$  with
\begin{itemize}
\item $\J^{\tC G}_{G}(\widetilde\chi)=G. \J^{\tC Z}_{CZ} (\widetilde\Omega'_B(\widetilde\chi))$ (central product) for every $\widetilde\chi\in\IBr_B(\widetilde CG)$,
   \item $\widetilde\Omega'_B( \IBr_B(\widetilde CG\mid\nu))=\IBr_B(\widetilde CZ\mid\nu)$ for every $\nu\in\IBr(\ZZ(\widetilde CG))$, and
    \item $\widetilde\Omega'_B(\widetilde\chi\lambda)=\widetilde\Omega'_B(\widetilde\chi)\Res^{\tC G}_{\tC Z}(\lambda)$ for every $\widetilde\chi\in\IBr_B(\widetilde CG)$, and $\lambda\in\IBr(\widetilde CG/G)$.
\end{itemize}
Also by Proposition~\ref{rela-G-CN} (iii),
$\IBr_B(\tC G)=\IBr(\tC G\mid \IBr_B(G))$ and $\IBr_B(G)=\IBr(G\mid \IBr_B(\tC G))$.

On $\IBr_B(\widetilde CG)$ the group $\tC G\rtimes D$ acts by conjugation and the group of linear Brauer characters $\IBr(\tC G/G)$ by multiplication.
Let $\widetilde{\mathcal G}$ be a transversal in $\IBr_B(\widetilde CG)$ with respect to these combined actions.
For every $\widetilde\chi\in \widetilde{\mathcal G}$ 
we let $\widetilde\chi'\in\IBr(\tG\mid \widetilde\chi)$ and
we fix a Brauer character $\chi_0\in \IBr(G\mid\widetilde\chi')$
with the properties from (ii) ~(then we have $\chi_0\in\IBr_{B}(G)$ automatically) and let $\mathcal G:=\{ \chi_0\mid\widetilde\chi\in\widetilde{\mathcal G} \}\subseteq\IBr_B(G)$ be the set formed by them.
Then $\mathcal G$ is a $(\tC G\rtimes D)$-transversal in $\IBr_B(G)$.

Since $\widetilde \Omega'_B$ is $\tC\rtimes D$-equivariant,
the set $\widetilde{\sC}:=\widetilde \Omega'_B(\widetilde{\mathcal G})$ is a transversal in $\IBr_B(\tC Z)$ with respect to the combined actions of $\IBr(\tC Z/CZ)$ and $\tC\rtimes D$.
As before we can associate to every $\widetilde\psi\in \widetilde{\sC}$ a Brauer character $\psi_0\in \IBr(CZ\mid\widetilde\psi)$ with the properties from (iii)~(then  $\psi_0\in\IBr_{B}(CZ)$ automatically holds).
Let these Brauer characters form the set $\sC$, which is a $(\tC\rtimes D)$-transversal in $\IBr_B(CZ)$.

For $\widetilde\chi\in\widetilde{\mathcal G}$ and $\chi_0\in \mathcal G \cap \IBr_B(G\mid\widetilde\chi)$, we define $\Omega_B(\chi_0):=\psi_0$ to be the unique element in $\sC\cap \IBr_B(CZ\mid \widetilde\psi)$, where $\widetilde\psi=\widetilde\Omega'_B(\widetilde\chi)\in \widetilde{\sC}$.
Now we prove that $(\tC Z\rtimes D)_{\psi_0}=(\tC Z\rtimes D)_{\chi_0}$.

By (\ref{sta-ibr}) we have
$$(\tC G)_{\chi_0}=\bigg (\bigcap\limits_{\{\lambda\in\IBr(\tC G/G)\ |\ \widetilde\chi\lambda=\widetilde\chi\}} \ker(\lambda)\bigg)\bigcap \J^{\tC G}_{G}(\widetilde\chi)$$
and
$$(\tC Z)_{\psi_0}= \bigg(\bigcap\limits_{\{\lambda\in\IBr(\tC Z/CZ)\ |\ \widetilde\psi\lambda=\widetilde\psi\}} \ker(\lambda)\bigg)\bigcap  \J^{\tC Z}_{CZ}(\widetilde\psi).$$
Also,
$$D_{\chi_0}= \{ d\in D\ | \ \widetilde\chi^d=\widetilde\chi\lambda\  \text{for some}\  \lambda \in \IBr(\tC G/G)  \}$$
and $$D_{\psi_0}= \{ d\in D\ | \ \widetilde\psi^d=\widetilde\psi\lambda\  \text{for some}\  \lambda \in \IBr(\tC Z/CZ)  \}.$$
Thus $(\tC G)_{\chi_0}=\tC_{\psi_0} G$, $D_{\chi_0}=D_{\psi_0}$.
So $(\tC \rtimes D)_{\psi_0}=(\tC\rtimes D)_{\chi_0}$.

Hence we can define a $\tC\rtimes D$-equivariant bijection 
$\Omega_B: \ \IBr_B(G)\to \IBr_B(CZ)$ by $$\Omega_B(\chi_0^x):=\Omega_B(\chi_0)^x\ \text{for every}\ x\in \tC\rtimes D\ \text{and}\ \ \chi_0\in \mathcal G.$$
Now we prove (\ref{equa-1}).
By~\cite[Lem~3.8]{SV16}, it suffices to let $\chi\in\mathcal G$.
Then we can establish a central isomorphism (\ref{equa-1}) by Proposition~\ref{twostepext} by taking $A$, $B$, $N$, $M$, $X$, $Y$, $H$, $L$, $\theta$, $\varphi$
to be
$\tC G\rtimes D$, $\tC Z\rtimes D$, $G$, $CZ$, $\tC G$, $\tC Z$, $G\rtimes D$, $CZ\rtimes D$, $\chi$,  $\Omega_B(\chi)$
respectively,
 and this completes the proof.
\end{proof}

\section{Simple groups of Lie type in non-defining characteristic}\label{SG-Lie}

Let $\bG$ be a 
simple linear algebraic group of simply connected type over an algebraic closure of $\F_p$  (where $p$ is a prime).
Let  $\mathbf B$ be a Borel subgroup of $\bG$ with maximal torus $\mathbf T$.
Let $\Phi$, $\Delta$ denote the set of roots and simple roots of $\bG$ determined by $\mathbf B$ and $\mathbf T$.
For a description of the Frobenius endomorphisms, we use the Chevalley generators $x_\alpha(t)$ ($t\in\overline{\F}_q$, $\alpha\in \Phi$) as in~\cite[Thm.~1.12.1]{GLS98}.

We recall the automorphisms of $\bG$ described as in~\cite[\S2]{MS16}.
Let $F_0:\bG\to\bG$ be the field endomorphism of $\bG$ given by $F_0(x_\alpha(t))=x_\alpha(t^p)$ for every $t\in\overline{\F}_p$ and $\alpha\in \Phi$.
Any length-preserving automorphism  $\tau$ of the Dynkin diagram associated to $\Delta$ and hence automorphism of $\Phi$ determines a graph automorphism $\gamma$ of $\bG$ given by $\gamma(x_\alpha(t))=x_{\tau(\alpha)}(t)$ for every $t\in\overline{\F}_p$ and $\alpha\in \pm\Delta$.
Note that such $\gamma$ commutes with $F_0$.

Let $r$ be the rank of $\ZZ(\bG)$ (as finite abelian group)
and $\mathbf Z\cong (\overline{\F}_p^\times)^r$
a torus of that rank with an embedding of $\ZZ(\bG)$.
We set $\widetilde{\bG}:=\bG\times_{\ZZ(\bG)} \mathbf Z$ the central product of $\bG$ with $\mathbf Z$ over $\ZZ(\bG)$.
Then $\widetilde{\bG}$ is a connected reductive group such that the natural map $\bG\hookrightarrow\widetilde{\bG}$ is a regular embedding (see~\cite[\S1.7]{GM20}).
From this, we can extend $F_0$ to a Frobenius endomorphism of $\widetilde{\bG}$ 
and $\gamma$  to an automorphism of $\widetilde{\bG}$ as in~\cite[p.~874]{MS16}.

 Let us consider a Frobenius endomorphism $F:=F_0^f\gamma$, with $\gamma$ a (possibly trivial) graph automorphism of $\bG$, leaving aside the case of types $^2{\mathsf B}_2$, $^2{\mathsf G}_2$ and $^2{\mathsf F}_4$.
Then $F$ defines an $\F_q$-structure on $\widetilde{\bG}$, where $q=p^f$.
The group of rational points $G=\bG^F$ and $\tG=\widetilde{\bG}^F$  are finite~(see~\cite[Thm.~21.5]{MT11}).
Note by construction that the order of $F_0$ as
automorphism of $G$ coincides with that of $F_0$ as automorphism of $\tG$.
The analogous statement also holds for any graph automorphism $\gamma$ and the
automorphisms of $\tG$ associated with it.

Let $D$ be the subgroup of $\Aut(G)$ generated by $F_0$ (here we identify $F_0$ with $F_0|_G$) and the graph automorphisms  commuting with $F$.
Then $\tG\rtimes D$ is well defined and induces all automorphisms of $G$, see~\cite[Thm.~2.5.1]{GLS98}. Concerning coprime action, note that automorphisms of $G$ of order prime to $|G|$ are therefore restrictions to $G$ of Frobenius endomorphisms of $\bG$.

The following was introduced in \cite{Sp12} in relation with the inductive McKay condition. We just saw the similar statement for Brauer characters in Theorem~\ref{criforlie}(ii).

\begin{defn}\label{Ainfty} When $\bG$ is simple simply connected, one says that $\bG^F$ satisfies the $A(\infty)$ condition if and only if 
 \begin{itemize}
 
\item[$A(\infty)$] every character $\chi\in\Irr(\bG^{F})$ has a ${\widetilde{\bG}}^{F}$-conjugate $\chi_0$ such that $({\widetilde{\bG}}^{F}\rtimes D)_{\chi_0}={\widetilde{\bG}}^{F}_{\chi_0}\rtimes D_{\chi_0}$ and $\chi_0$ extends to $\bG^{F}\rtimes D_{\chi_0}$.	 
\end{itemize}
\end{defn}

This was proved recently to be satisfied in all types. 

\begin{thm}\label{split-char} Any group $\bG^F$ with $\bG$ simple simply connected satisfies $A(\infty)$.	
\end{thm}

\begin{proof}
This is proved in~\cite[Thm.~4.1]{CS17a},~\cite[Thm.~3.1]{CS17b},~\cite[Thm.~B]{CS19} and \cite[Thm.~A]{Sp22}.
\end{proof}

We denote by
${\bG}^{*}$ and ${\widetilde{\bG}}^*$  the dual groups of $\bG$ and $\widetilde{\bG}$ respectively with corresponding Frobenius  endomorphisms also denoted by $F$.
Many constructions for algebraic groups have a convenient formulation in terms of generic groups.
Suppose that $\mathbb G$ is a complete root datum of $(\bG, F)$ (in the sense of~\cite[\S2]{BM92}) so that $G=\bG^F=\mathbb G(q)$.
Then we have a complete root datum $\mathbb G^*$ for $(\bG^*,F)$~(see~\cite[\S1]{BMM93}).
Also, ${\bG^*}^F=\mathbb G^*(q)$.

Let $B\le\Aut(G)$ such that $\gcd(|G|,|B|)=1$, then by~\cite[\S2]{MNS15}, $B$ is $\Aut(G)$-conjugate to some subgroup of the
field automorphisms of $G$.
By~\cite[Rmk.~6.2]{SV16}, when considering the  (iBG)  condition, we can assume  that $B$ is generated by some power of $F_0$ and then $B$ is cyclic.  
Then a bijection as required in Definition~\ref{iBGC} (I) has been given in~\cite[Thm.~3.1]{NST17}.
However, it seems hard to construct the central isomorphism (\ref{equaiton-22}) in general then.
Therefore, we start with considering the irreducible ordinary characters.
Let $C=\C_G(B)$ and $\tC=\C_{\tG}(B)$.

\begin{thm}\label{corr-ord-char}
	Suppose that $\mathbb G$ is not of type ${}^3{\mathsf D}_4$. 
Let $B=\langle F_0^e \rangle$ such that $\gcd(|\tG|,|B|)=1$, and let $C=\C_G(B)$, $\tC=\C_{\tG}(B)$.
Then there is a $D$-equivariant bijection $\widetilde{\Xi}_B: \Irr_B(\tG)\to \Irr(\tC)$ such that 
\begin{enumerate}[\rm(i)]
	\item $\widetilde{\Xi}_B(\Irr_B(\tG\mid \nu))=\Irr(\tC\mid \nu)$ for every $\nu\in\Irr(\ZZ(\tG)\cap \tC)$, and
\item $\widetilde{\Xi}_B(\widetilde\chi\lambda)=\widetilde{\Xi}_B(\widetilde\chi)\Res^{\tG}_{\tC}(\lambda)$ for every $\widetilde\chi\in\Irr_B(\tG)$ and $\lambda\in\Irr_B(\tG/G)$.
\end{enumerate}
\end{thm}

\begin{proof}
Since $(\bG, F)$ is not of type ${}^3{\mathsf D}_4$, $\gamma$ is of order $1$ or $2$. 	
We recall that, as an automorphism of $G$ (or $\tG$), $F_0$ is of order $f$ if $\gamma$  is of order $1$
while $F_0$ is of order $2f$ if $\gamma$  is of order $2$.

First assume that $\gamma$ is trivial. 
Let $q_0=p^e$ and $F_1:=F_0^e$.
Then $\tC=\widetilde{\bG}^{F_1}$ and $C=\bG^{F_1}=\mathbb G(q_0)$.  Here we have $F=F_1^{f/e}$.	
Of course $B$ is generated by $F_1$.

Now let $\gamma$ be of order 2. Then $\gamma=F_0^f$ as automorphisms of $\tG$.
From $\gcd(|G|,|B|)=1$ we know that $\frac{2f}{e}$ is odd and then $e$ is even.
So for $g\in\tG$,  one has
$g\in \tC$ $\Leftrightarrow$ $g=F_0^e(g)=(F_0^\frac{e}{2})^2(g)$ $\Leftrightarrow$ 
$\gamma(g)=F_0^f(g)=(F_0^\frac{e}{2})^\frac{2f}{e} (g)=F_0^\frac{e}{2}(g)$ $\Leftrightarrow$ 
$\gamma F _0^\frac{e}{2}(g)=g$.
Let $q_0=p^\frac{e}{2}$ and $F_1:=F_0^\frac{e}{2}\gamma$.
Then $\tC=\widetilde{\bG}^{F_1}$ and $C=\bG^{F_1}=\mathbb G(q_0)$.
It can be checked that $F=F_1^{\frac{2f}{e}}$.	
As automorphisms of $\widetilde\bG^F$, we have
$F_1=F_0^{\frac{e}{2}}F_0^{f}=(F_0^e)^{\frac{\frac{2f}{e}+1}{2}}$.
Recall that $F_0^e$ has order $\frac{2f}{e}$, which is coprime with $\frac{\frac{2f}{e}+1}{2}$.
Therefore, $F_1$ generates $B$.

In both cases, $C=\bG^{F_1}=\mathbb G(q_0)$ and  $\tC=\widetilde{\bG}^{F_1}$ so that $q$ is a power of $q_0$ and $F$ is a power of $F_1$.
In addition, $B$ is generated by $F_1$.
Thus an irreducible character of $\widetilde\bG^F$ is $B$-invariant if and only if it is $F_1$-invariant.

The required bijection  $\widetilde{\Xi}_B$ is indeed constructed in the proof of~\cite[Thm.~3.7]{CS19}.
But for convenience, we recall some details. 
First note that ${\bG^*}^{F}=\mathbb G^*(q)$ and ${\bG^*}^{F_1}=\mathbb G^*(q_0)$.
By the Jordan decomposition, there is a bijection $\Psi_{\widetilde{\bG},F}$ from the set of
$(\widetilde{\bG}^*)^F$-conjugacy classes of pairs $(s,\phi)$ with $s\in (\widetilde{\bG}^*)^F_{\mathrm{ss}}$ and $\phi\in\mathcal E(\C_{\widetilde{\bG}^*}(s)^F,1)$
to $\Irr(\widetilde{\bG}^F)$.
Here $(\widetilde{\bG}^*)^F_{\mathrm{ss}}$ denotes the set of semisimple elements of $\widetilde{\bG}^*$ which are $F$-invariant.
In addition, we choose $\Psi_{\widetilde{\bG},F}$ as in~\cite[Thm.~7.1]{DM90},
and then $\Psi_{\widetilde{\bG},F}$ is $F_1$-equivariant (see also \cite[Thm.~3.1]{CS13}).
Similarly we have the bijection $\Psi_{\widetilde{\bG},F_1}$ for the parametrization of irreducible characters of $\widetilde{\bG}^{F_1}$.

Since centralizers of semisimple elements of $\widetilde{\bG}^*$ are connected by~\cite[Rmk.~11.2.2~(ii)]{DM91}, there are well-known bijections $$(\widetilde{\bG}^*)^F_{\mathrm{ss}}/\sim_{(\widetilde{\bG}^*)^F}\longleftrightarrow (\widetilde{\bG}^*)^F_{\mathrm{ss}}/\sim_{\widetilde{\bG}^*}\longleftrightarrow (\widetilde{\bG}^*_{\mathrm{ss}}/\sim_{\widetilde{\bG}^*})^F$$~(see~\cite[4.3.6]{Ge03}).
Here, $\sim_{\widetilde{\bG}^*}$ (or $\sim_{(\widetilde{\bG}^*)^F}$, resp.) denotes the relation on $\widetilde{\bG}^*$ (or $(\widetilde{\bG}^*)^F$, resp.) of conjugacy.
They are $F_1$-equivariant and hence 
$$(\widetilde{\bG}^*)^{F_1}_{\mathrm{ss}}/\sim_{(\widetilde{\bG}^*)^{F_1}}\xrightarrow{\sim}((\widetilde{\bG}^*)^F_{\mathrm{ss}}/\sim_{(\widetilde{\bG}^*)^F})^{F_1}$$
by the obvious map.
Using~\cite[Cor.~3.6]{CS19}, there is a $D$-equivariant bijection
$$\Upsilon: \mathcal E(\C_{\widetilde{\bG}^*}(s)^{F_1},1)\longrightarrow\mathcal E(\C_{\widetilde{\bG}^*}(s)^F,1)^{F_1},$$
which is equivariant for algebraic automorphisms of $\widetilde{\bG}^*$ commuting with $F_1$.
Therefore $F_1$-stable $(\widetilde{\bG}^*)^F$-classes of pairs $(s,\phi)$ correspond to $(\widetilde{\bG}^*)^{F_1}$-classes of pairs $(s,\Upsilon^{-1}(\phi))$ with $s\in(\widetilde{\bG}^*)^{F_1}_{\mathrm{ss}}$ and $\phi\in\mathcal E(\C_{\widetilde{\bG}^*}(s)^F,1)^{F_1}$.

Therefore, we have a bijection
$\widetilde{\Omega}_B:\Irr_B(\widetilde{\bG}^F)\to\Irr(\widetilde{\bG}^{F_1})$ such that
$$\widetilde{\Omega}_B^{-1}(\Psi_{\widetilde{\bG},F_1}(s,\phi))=\Psi_{\widetilde{\bG},F}(s,\Upsilon(\phi)).$$
For $\sigma\in D$, we can define a dual $\sigma^*\in\Aut({{\bG}^*}^F)$ in the sense of~\cite[Def.~2.1]{CS13}.
Note that the dual of field (or graph) automorphisms are also field (or graph) automorphisms and they have similar forms.
By~\cite[Thm.~7.1~(vi)]{DM90}~(or~\cite[Thm.~3.1]{CS13}),
$\Psi_{\widetilde\bG,F}(s,\phi)^{\sigma^{-1}}=\Psi_{\widetilde\bG,F}(\sigma^*(s),\sigma^*(\phi))$.
From this, $\widetilde{\Omega}_B$ is $D$-equivariant.
Furthermore, (i) follows 
by choosing suitable $\hat s$~(see~\cite[(8.14)]{CE04}),
since all the characters in the Lusztig series $\mathcal E(\widetilde\bG^F,s)$ lie over the same character of $\ZZ(\widetilde\bG^F)$ which is the restriction of $\hat s$.

Now we prove (ii).
Let $\hat z_{F}:=\Psi_{\widetilde\bG,F}(z,1_{(\widetilde\bG^*)^F})$ for $z\in \ZZ(\widetilde\bG^*)^F$
and $\hat z_{F_1}:=\Psi_{\widetilde\bG,F_1}(z,1_{(\widetilde\bG^*)^{F_1}})$ for $z\in \ZZ(\widetilde\bG^*)^{F_1}$.
Then $\Irr(\widetilde\bG^F/\bG^F)=\{ \hat z_{F}\mid z\in \ZZ(\widetilde\bG^*)^F \}$
and $\Irr(\widetilde\bG^{F_1}/\bG^{F_1})=\{ \hat z_{F_1}\mid z\in \ZZ(\widetilde\bG^*)^{F_1} \}$.
Note that $\Irr_B(\widetilde\bG^F/\bG^F)=\{ \hat z_{F}\mid z\in \ZZ(\widetilde\bG^*)^{F_1} \}$ and
$\Res^{\bG^F}_{\bG^{F_1}}(\hat z_F)=\hat z_{F_1}$.
So (ii) follows by~\cite[Thm.~7.1~(iii)]{DM90}.
\end{proof}

From now on we assume that $\ell\ne p$.
According to Conjecture~\ref{GeckC} recalled in our Introduction it is expected that all $\ell$-blocks of a (quasi-simple) groups of Lie type have a unitriangular decomposition matrix.
Moreover, we expect that the corresponding basic set can be chosen stable under the action of the automorphism group~(see for example \cite[Thm.~2.5]{De17} and~\cite[4.7 and 4.8]{Ma17}).

\begin{defn}
Let $H$ be a finite group.
An injective map $\Theta:\IBr(H)\to\Irr(H)$ is called a \emph{unitriangular map}, if there exists a suitable ordering $\le$ of $\IBr(H)$ such that $\Dec_{\Theta(\IBr(H)),\IBr(H)}$ is unitriangular with respect to $\le$ and the ordering of $\Theta(\IBr(H))$ corresponding to $\le$ via $\Theta$.	
\end{defn}

We propose the following hypothesis.

\begin{hyp}\label{unitri}
Let $F:\bG\to \bG$ be a Frobenius endomorphism
endowing $\bG$ with an $\F_{q}$-structure so that $\mathbb G$ is the complete root datum of $(\bG,F)$~(which implies $\bG^{F}=\mathbb G(q)$).	
Suppose that there is a unitriangular map $\widetilde\Theta_q:\IBr({\widetilde{\bG}}^{F})\to\Irr({\widetilde{\bG}}^{F})$ such  that 
\begin{enumerate}[(i)]
\item 	$\widetilde\Theta_q(\IBr({\widetilde{\bG}}^{F}))$ is $\Irr(\tbG^F/\bG^F)_{\ell'}$-stable (note that $\Irr(\tbG^F/\bG^F)$ acts on $\Irr(\tbG^F)$ by multiplication) and $D$-stable,
\item
there is a unitriangular map $\Theta_q:\IBr(\bG^F)\to\Irr(\bG^F)$ such that 
$\Theta_q(\IBr(\bG^F\mid\widetilde\chi))=\Irr(\bG^F\mid \widetilde\Theta_q(\widetilde\chi))$ for every $\widetilde\chi\in\IBr(\tbG^F)$, and
\item $\widetilde\Theta_q$ is compatible with Frobenius endomorphisms, i.e.,  for $q_0\mid q$ with $q=q_0^e$ satisfying $\gcd(e,|\bG^{F}|)=1$, we have
$$\tilde{\Xi}(\widetilde\Theta_q(\IBr({\widetilde{\bG}}^{F}))^{F(q_0)})=\widetilde\Theta_{q_0}(\IBr({\widetilde{\bG}}^{F(q_0)}))$$ where
$F(q_0):\bG\to \bG$ is a Frobenius endomorphism
endowing $\bG$ with an $\F_{q_0}$-structure so that $\mathbb G$ is also the complete root datum of $(\bG,F(q_0))$,
$F$ is a power of $F(q_0)$ and $\widetilde\Xi$ is as in Theorem~\ref{corr-ord-char}.
\end{enumerate}
\end{hyp}

We say Hypothesis~\ref{unitri} holds for $\mathbb G$ if it is true for any $q$.

\begin{rmk}
According to \cite[Thm.~2.5]{De17}, condition (ii) of Hypothesis~\ref{unitri} can be replaced by the following: 
$|\IBr(\bG^{F}\mid \widetilde\chi)|_\ell=|\Irr(\bG^{F}\mid \widetilde\Theta_q(\widetilde\chi))|_\ell$ (which imples that $|\IBr(\bG^{F}\mid \widetilde\chi)|=|\Irr(\bG^{F}\mid \widetilde\Theta_q(\widetilde\chi))|$ in the situation of Hypothesis~\ref{unitri} (i))
for every $\widetilde\chi\in\IBr({\widetilde{\bG}}^{F})$.
\end{rmk}

Recall that according to~\cite[Thm.~1.7.15]{GM20}, $\Res^{{\widetilde{\bG}}^{F}}_{{\bG}^{F}}(\widetilde \chi)$ is multiplicity-free
for every $\widetilde \chi\in\Irr({\widetilde{\bG}}^{F})\cup\IBr({\widetilde{\bG}}^{F})$.
Then as said in Section 2.2 above every $\chi\in\Irr(\bG^F)\cup\IBr(\bG^F)$ extends to its stabilizer in ${\widetilde{\bG}}^{F}$.

\begin{thm}\label{mainthm}
Keep the setup of Theorem~\ref{corr-ord-char}, 
let $\ell\nmid q$,
and assume further that Hypothesis~\ref{unitri} holds for $\mathbb G$.
If $S=\bG^F/\ZZ(\bG^F)$ is simple, then the  (iBG)  condition (see Definition~\ref{iBGC}) holds for $S$ and the prime $\ell$.
\end{thm}

\begin{proof}
Keep the notation occuring in Theorem~\ref{corr-ord-char}. To verify the (iBG) condition,  we use Theorem~\ref{criforlie}.
For Hypothesis~\ref{unitri}, one can apply \cite[2.3]{De17} to both the action of $D$ and of $\Irr(\tbG^F/\bG^F)_{\ell '}$, and get that 
\begin{enumerate}[\rm(i)]
	\item $\widetilde\Theta_q(\lambda^0 \widetilde\chi)=\lambda \widetilde\Theta_q(\widetilde\chi)$ for every $\widetilde\chi\in\IBr(\tbG^F)$ and $\lambda\in \Irr(\tbG^F/\bG^F)_{\ell'}$,
	\item 	 $\widetilde \Theta_q: \IBr(\tbG^F)\to \widetilde{\mathcal B}$ is a $D$-equivariant bijection where $\widetilde{\mathcal B}=\widetilde \Theta(\IBr(\tbG^F))$,
	\item  $\Theta_q:\IBr(\bG^F)\to \Irr(\bG^F\mid \widetilde{\mathcal B})$ a $D$-equivariant bijection.
\end{enumerate}
Similar properties hold for groups $\bG^{F_1}=\mathbb G(q_0)$ and $\tbG^{F_1}$.
From this, to verify the conditions of Theorem~\ref{criforlie}, we can transfer to ordinary characters.

Condition (i)(a) of Theorem~\ref{criforlie} can be checked case by case while (i)(b) and (c) follow by Clifford theory and~\cite[Thm.~B]{Ge93} (or \cite[Thm.~1.7.15]{GM20}).
By $A(\infty)$ which is ensured by Theorem~\ref{split-char}, Hypothesis~\ref{unitri},  and the arguments of Remark \ref{RemIBr},
conditions (ii) and (iii) of Theorem~\ref{criforlie} are satisfied.
Applying Theorem~\ref{corr-ord-char} and by Hypothesis ~\ref{unitri} again,  we can deduce Theorem~\ref{criforlie}~(iv).
Here, (b) and (c) are obvious.
Condition (a) can be deduced by computing the stabilizers of the corresponding characters in the basic sets from Hypothesis ~\ref{unitri}. Analogously as in the proof of Theorem~\ref{criforlie}, the map
$\Res^{\tG/G}_{\tC G/G}:\Irr_B(\tG/G)\to\Irr(\tC G/G)$ is bijective.
Then (a) follows by~(\ref{sta-irr}), Theorem~\ref{corr-ord-char}~(ii) and a similar argument as in the proof of Theorem~\ref{criforlie}.
\end{proof}

As an application of Theorem~\ref{mainthm}, we establish the (iBG) condition for simple groups of types $\mathsf A$ and $^2\mathsf A$ now,
using the unitriangular basic sets of $\SL_n(q)$ and $\SU_n(q)$ given in~\cite{De17}.

\begin{thm}\label{for-type-A}
	The simple groups $\PSL_n(q)$ and $\PSU_n(q)$ satisfy the (iBG) condition.
\end{thm}

\begin{proof}
By~\cite[Thm.~5.1]{NST17}, it suffices to consider the non-defining characteristic.
Thanks to Corollary~\ref{exceptional-covering},
we assume that $\SL_n(\pm q)$ is the universal covering group of $\PSL_n(\pm q)$.
This theorem follows by Theorem~\ref{split-char} and~\ref{mainthm} while Hypothesis~\ref{unitri}  can be verified as follows.
We use the notation from Theorem~\ref{corr-ord-char} and its proof.
First, Hypothesis~\ref{unitri}~(i) and (ii) hold by the construction of basic sets in~\cite{De17} immediately.
For (iii), we should recall the details of the construction.

Let $\tG=\GL_n(\eps q)$, where $\eps=\pm 1$. As usual, $\GL_n(-q)$ denotes the general unitary group.
The parametrization of $\IBr(\tG)$ is given in~\cite{KT09} for $\eps=1$, called the admissible symbols, while the case  $\eps=-1$ was considered in~\cite{De17}.
Note that an exposition can be found in~\cite[Rmk.~3.4]{Fe19}.
We will use the admissible symbols to parametrize the irreducible characters and then give a  parametrization for the basic set consructed in~\cite{De17}.

For a partition $\mu=(\mu_1,\mu_2,\ldots)$, denote $|\mu|=\mu_1+\mu_2+\cdots$ and write $\mu'$ for the transposed partition.
Set $\Delta(\mu)=\mathrm{gcd}(\mu_1,\mu_2,\ldots)$.	
Let  $h\mid \Delta(\mu')$.
We rewrite $\mu_i=(\mu_1^{t_1},\mu_2^{t_2},\ldots)$ and then
we set $\mu/h=(\mu_1^{t_1/h},\mu_2^{t_2/h},\ldots)$.

For $\sigma\in\overline{\F}_p^\times$, 
we set $$[\sigma]_{\eps q}=\{\ \sigma, \sigma^{\epsilon q}, \sigma^{(\epsilon q)^2},\ldots, \sigma^{(\epsilon q)^{\mathrm{deg}_{\eps q}(\sigma)-1}}    \ \},$$
where $\mathrm{deg}_{\eps q}(\sigma)$ is the minimal integer $d$ such that $\sigma^{(\epsilon q)^d-1}=1$.

An \emph{$(n,\eps q)$-admissible tuple} is a tuple 
\begin{equation}\label{admtup}
(([\sigma_1]_{\eps q},\mu^{(1)}),\dots,([\sigma_a]_{\eps q},\mu^{(a)}))
\addtocounter{thm}{1}\tag{\thethm}
\end{equation}
 of pairs, where $a\geq 1$, $\sigma_1,\dots,\sigma_a\in \overline{\F}_p^\times$, and $\mu^{(1)}, \dots, \mu^{(a)}$ are partitions such that
\begin{itemize}
	\item $[\sigma_i]_{\eps q}\ne [\sigma_j]_{\eps q}$ for all $i\ne j$, and 
	\item $\sum\limits^{a}_{i=1}\mathrm{deg}_{\eps q}(\sigma_i)|\mu^{(i)}|=n$.
\end{itemize} 	
	
The equivalence class of an $(n,\eps q)$-admissible tuple (\ref{admtup}) up to permutations of the pairs $$([\sigma_1]_{\eps q},\mu^{(1)}),\ldots,([\sigma_a]_{\eps q},\mu^{(a)})$$ is called an \emph{$(n,\eps q)$-admissible symbol} and is denoted as 
\begin{equation}\label{admsym}
\mathfrak s=[([\sigma_1]_{\eps q},\mu^{(1)}),\dots,([\sigma_a]_{\eps q},\mu^{(a)})].
\addtocounter{thm}{1}\tag{\thethm}
\end{equation} This is clearly in bijection with $\tG$-classes of pairs $(s,\phi)$ where $s\in\tG_{\rm{ss}}$ and $\phi\in\mathcal{E}(\mathrm{C}_{\tG}(s),1)$ thanks to the parametrization of unipotent characters of GL$_d(\pm q)$ by partitions of $d$. Then Jordan decomposition of characters as recalled in the proof of Theorem~\ref{corr-ord-char} implies that
 $(n,\eps q)$-admissible symbols is a labeling set for $\Irr(\tG)$.
Denote by $\chi^{\tG}_{\mathfrak s}$ the irreducible  character corresponding to the $(n,\eps q)$-admissible symbol $\mathfrak s$.
	
For the $(n,\eps q)$-admissible symbol (\ref{admsym}) we define 
$$\mathfrak s^k=[([\sigma_1^k]_{\eps q},\mu^{(1)}),\dots,([\sigma_a^k]_{\eps q},\mu^{(a)})],$$
if $k=-1$ or $k$ is a power of $p$.
Thus by~\cite[Prop.~3.5]{De17},  $(\chi^{\tG}_{\mathfrak s})^{F_0}=\chi^{\tG}_{\mathfrak s^p}$ and $\chi^{\tG}_{\mathfrak s^{-1}}$ is the image of the character  $\chi^{\tG}_{\mathfrak s}$ under a certain graph automorphism. Moreover each linear character of $\tG/G$ acts by the associated scalar multiplication on the $\sigma_i$'s.

Let $B=\langle F_0^e\rangle$ such that $\gcd(|B|,|G|)=1$.
We let $q_0$ and $F_1$ be as in the proof of Theorem~\ref{corr-ord-char}.
Then $\tC=\GL_n(\eps q_0)$.
In addition, $\chi^{\tG}_{\mathfrak s}$ is $B$-invariant if and only if it is $F_1$-invariant,
and if and only if for any  $1\le i \le a$, there exists some $1\le j \le a$ such that $\sigma_i^{\eps q_0}=\sigma_j$ and $\mu^{(i)}=\mu^{(j)}$. 
From this, if $\chi^{\tG}_{\mathfrak s}$ is $B$-invariant, one may assume that
$\mathfrak s$ is the class of the $(n,\eps q)$-admissible symbol of pairs 
$([\sigma_1]_{\eps q},\mu^{(1)}),$
$([\sigma_1^{\eps q_0}]_{\eps q},\mu^{(1)}),\dots,$
$([\sigma_1^{(\eps q_0)^{l_1}}]_{\eps q},\mu^{(1)}),$ $\dots,$
$([\sigma_s]_{\eps q},\mu^{(s)}),$ 
$([\sigma_s^{\eps q_0}]_{\eps q},\mu^{(s)}),\dots,$
$([\sigma_s^{(\eps q_0)^{l_s}}]_{\eps q},$ $\mu^{(s)})$
for some integers $s,l_1,\ldots,l_s$ all $\geq 1$.
Then by the construction of $\Xi$ in the proof of Theorem~\ref{corr-ord-char},
$\Xi_B(\chi^{\tG}_{\mathfrak s})=\chi^{\tC}_{\mathfrak s_0}$, where
$\mathfrak s_0$ is the $(n,\eps q_0)$-admissible symbol of pairs
$([\sigma_1]_{\eps q_0},\mu^{(1)}),$
$\dots,$
$([\sigma_s]_{\eps q_0},\mu^{(s)})$.

Now we describe the basic set constructed in~\cite[\S4.4]{De17} in our notation above.
Let $\widetilde{\mathcal E}_{\eps q}$ be the set consisting of $(n,\eps q)$-admissible symbols~(\ref{admsym}) such that $\sigma_1$, $\dots$, $\sigma_a$ are $\ell'$-elements.
For $\mathfrak s\in \widetilde{\mathcal E}_{\eps q}$ we go through the following steps
\begin{itemize}
\item if $\ell^d=\gcd\big(q-\eps, \Delta((\mu^{(1)})'),\ldots, \Delta((\mu^{(a)})')\big)_\ell\ne 1$ then take $w$ to be an element in $\overline{\F}_p^\times$ having order $\ell^d$,
\item let $\mathfrak s'$ be the $(n,\eps q)$-admissible symbol of pairs
$([\sigma_1]_{\eps q},\mu^{(1)}/\ell^d),$
$([\sigma_1 w]_{\eps q},\mu^{(1)}/\ell^d),\dots,$
$([\sigma_1 w^{\ell^a-1}]_{\eps q},\mu^{(1)}/\ell^d),$ $\dots,$
$([\sigma_a]_{\eps q},\mu^{(a)}/\ell^d),$ 
$([\sigma_a w]_{\eps q},\mu^{(a)}/\ell^d),\dots,$
$([\sigma_a w^{\ell^a-1}]_{\eps q},$ $\mu^{(a)}/\ell^d)$, and
\item replace $\mathfrak s$ by $\mathfrak s'$ in $\widetilde{\mathcal E}_{\eps q}$.
\end{itemize}
From this we obtain a new set $\widetilde{\mathcal E}'_{\eps q}$ consisting of $(n,\eps q)$-admissible symbols.
Then $\widetilde{\mathcal I}'_{\eps q}:=\{\chi^{\tG}_{\mathfrak s}\mid \mathfrak s\in\widetilde{\mathcal E}'_{\eps q}\}$ forms a basic set for $\tG$ with a unitriangular decomposition matrix while ${\mathcal I}'_{\eps q}:=\Irr(G\mid \widetilde{\mathcal I}'_{\eps q})$ has the same properties for $G$. We then get bijections $\widetilde{\Theta}_{ q} \colon \IBr(\tG)\to \widetilde{\mathcal I}'_{\eps q}$ and ${\Theta}_{ q} \colon \IBr(G)\to {\mathcal I}'_{\eps q}$. Now Hypothesis~\ref{unitri}(i), as a property of $ \widetilde{\mathcal I}'_{\eps q}$ is an easy consequence of what we have recalled about action of linear characters and automorphisms, while Hypothesis~\ref{unitri}(ii) is the main unitriangularity property of~\cite{De17}.
Finally, $\Xi_B((\widetilde{\mathcal I}'_{\eps q})^B)=\widetilde{\mathcal I}'_{\eps q_0}$ can be checked directly and this is  Hypothesis~\ref{unitri}~(iii).
Thus we have completed the proof.
\end{proof}

We can now restate Theorem~\ref{main-thm1} as follows.

\begin{thm}
	If Hypothesis~\ref{unitri} holds for types ${\mathsf B}_n$, ${\mathsf C}_n$, ${\mathsf D}_n$, ${}^2 {\mathsf D}_n$, ${\mathsf E}_6$, ${}^2 {\mathsf E}_6$ and ${\mathsf E}_7$ whenever $\ell\nmid q$,
	then Conjecture~\ref{conj-corr} is true.
\end{thm}

\begin{proof}
	It was proved  that the  (iBG)  condition holds for every finite simple group with cyclic outer automorphism group by~\cite[Thm.~4.4]{NST17},
	for every finite simple group not of Lie type by~\cite[Cor.~4.6 and Prop.~4.7]{NST17}, and for every finite simple group of Lie type in defining characteristic by~\cite[Thm.~5.1]{NST17}.
	By Theorem~\ref{for-type-A} and the structure of the automorphism groups of finite simple groups (see, for example,~\cite[Thm.~2.5.12]{GLS98}),  it remains to consider the simple groups of Lie types  ${\mathsf B}_n$, ${\mathsf C}_n$, ${\mathsf D}_n$, ${}^2 {\mathsf D}_n$, ${\mathsf E}_6$, ${}^2 {\mathsf E}_6$ and ${\mathsf E}_7$ in non-defining characteristic.
	Also, the exceptional covering groups are verified to satisfy the (iBG) condition by Corollary~\ref{exceptional-covering}.  
	
	Thanks to Theorem~\ref{mainthm}, we can apply Theorem~\ref{main-thm-SV16} now and this implies our
	theorem.
\end{proof}

\section{More results on the inductive Brauer--Glauberman condition}\label{application}

First we give an application of Theorem~\ref{criforlie}.

\begin{thm}\label{typeC-2}
	The (iBG) condition holds for simple groups $\mathrm{PSp}_{2n}(q)$ ($n\geq 2$) and the prime~2.
\end{thm}

\begin{proof}
	Thanks to~\cite[Thm.~5.1]{NST17}, we can assume $q$ to be odd.
Let $S=\mathrm{PSp}_{2n}(q)$.	
Let $\bG=\Sp_{2n}(\barF_q)$ and $\tbG=\CSp_{2n}(\barF_q)$.
Keep the notation preceding Theorem \ref{corr-ord-char}.
	Let $q=p^f$ and we take $F=F_0^f$, then $G=\bG^F=\Sp_{2n}(q)$ and $\tG=\tbG^F=\CSp_{2n}(q)$.
	Also, $G$ is the universal covering group of $S$.
	Chaneb~\cite{Ch21} has proved that there is a unitriangular basic set for the
	unipotent~2-block (principal) of $G$.
Based on this result, the paper~\cite{FM20} gives a labelling
	set for $\IBr(G)$, together with the action of automorphisms. 
	Now we recall that labelling set.
	
	Let $\cF$ be the set defined in~\cite[\S1]{FS89} consisting
	of the polynomials serving as elementary divisors for all semisimple elements of $G^*$. Note that $G^*=\mathrm{SO}_{2n+1}(q)$.
	For $\Gamma\in\cF$, let $\delta_\Gamma=1$ if $\Gamma=x-1$ or $x+1$, and $\delta_\Gamma$ be the half of the degree of $\Gamma$ if $\Gamma\notin\{x-1,x+1\}$. Also, we define $\vare_\Gamma$ as in~\cite[(1.9)]{FS89}.
	For a semisimple element $s$ of $G^*$, if $\Gamma\in \cF$ is an elementary divisor, then we denote by $m_\Gamma(s)$ the multiplicity of $\Gamma$ in $s$. For convenience, if $\Gamma\in\cF$ is not an elementary divisor of $s$, then we let $m_\Gamma(s)=0$.
	If $s$ is of $2'$-order, 
	we also let $w_{x-1}(s)$ be the integer such that $m_{x-1}(s)=2 w_{x-1}(x)+1$.
	Then
	$\C_{G^*}(s)^*\cong \Sp_{2w_{x-1}(s)}(q)\times \prod_{\Gamma} \GL_{m_{\Gamma}(s)}(\vare_\Gamma q^{\delta_\Gamma})$, where $\Gamma$ runs through the elementary divisors of $s$ in $\cF\setminus \{ x-1,x+1\}$.
	By~\cite[3.1]{FM20}, for every semisimple $2'$-element $\tilde s\in\tG^*$ ($s\in G^*$), $\cE_2(\tG,\tilde s)$ (resp. $\cE_2(G,s)$) is a 2-block of $\tG$ (resp. $G$).

	We recall the parametrization from~\cite{FM20}  on the unipotent Brauer characters
	of $G$ as follows. Let $\scU(n)$ be the set of maps
	$\bm:\Z_{\ge 1}\to \Z_{\ge 0}$ such that $\sum_{j\ge 1} j\bm(j)=2n$ and
	$\bm(j)$ is even if $j$ is odd. Here, $\bm$ is counted $2^{k_{\bm}}$-times,
	where $k_{\bm}=|\{ j\mid j\text{ is even and }\bm(j)\ne 0 \}|$.
Then 
$\scU(n)$ is a parametrization of unipotent classes of $G$.
By \cite[Prop.~4.2]{FM20}, there is an $\Aut(G)$-equivariant bijection between the  unipotent classes of $G$ and $\IBr(\cE_2(G,1))$.
So $\scU(n)$ is a labelling set for  $\IBr(\cE_2(G,1))$.
	Let $\scU_1(n)$ be the subset of $\scU(n)$ consisting of those $\bm$ such that
	$\bm(j)$ is even for all $j$. Again, an element $\bm\in\scU_1(n)$ is counted
	$2^{k_{\bm}}$-times in $\scU_1(n)$. 
	Let $\scU_2(n)=\scU(n)\setminus\scU_1(n)$.
	We let $\mathscr P(n)$ be the set of partitions of $n$.
	For a semisimple $2'$-element $s$ of $G^*$, we define 
	$\Upsilon(s):=\scU(w_{x-1}(s))\times\prod\limits_{\Gamma\in \cF\setminus \{ x-1,x+1\}}\mathscr P(m_{\Gamma}(s))$.
	Then $\Upsilon(s)$ is a labelling set for the irreducible 2-Brauer characters in the unipotent block of $\C_{G^*}(s)^*$, and then by~\cite[Prop. 4.5]{FM20},
	a labelling set for $\IBr(\cE_2(G,s))$.
	Let $\Upsilon(G)$ be the set of $G^*$-conjugacy classes of pairs $(s,\mu)$, where $s$ is a semisimple $2'$-element of $G^*$ and $\mu\in\Upsilon(s)$.
	Then $\Upsilon(G)$
	is a labelling set for $\IBr(G)$.
		
	The actions of automorphisms on $\IBr(G)$ are given in~\cite[Cor.~4.3 and Prop.~4.5]{FM20} and we recall them as follows.
	Let $\Upsilon_1(s)=\scU_1(w_{x-1}(s))\times\prod\limits_{\Gamma\in \cF\setminus \{ x-1,x+1\}}\mathscr P(m_{\Gamma}(s))$
	be the subset of $\Upsilon(s)$.
	Then a Brauer character $\chi\in\IBr(G)$ corresponding to $(s,\mu)\in \Upsilon(G)$ is $\tG$-invariant if and only if 
	$\mu\in\Upsilon_1(s)$.
	If $\mu\notin\Upsilon_1(s)$ and  $\chi^g$  corresponds to $(s',\mu')$, then $s'$ is $G^*$-conjugate to $s$, $s'_\Gamma=s_\Gamma$ for $\Gamma\ne x-1$ and $\mu_{x-1}$ and $\mu'_{x-1}$ are labelled by the same element in $\scU_2(w_{x-1}(s))$.
	Define $F_0^*$ as in~\cite[Def.~2.1]{CS13}. It is a field automorphism of $\tG^*$. 
	We note that ${F_0^*}^{-1}$ acts on $\cF$, see for instance~\cite[Prop.~5.5]{FM20}.
	Thus $\chi^{F_0}$ corresponds to $(s',\mu')$, where $s'={F_0^*}^{-1}(s)$ and $\mu'_{{F_0^*}^{-1}(\Gamma)}=\mu_{\Gamma}$.
	We denote by ${F_0^*}^{-1}(\mu)$ this $\mu'$. 
	Thus the action of $F_0$ on $\IBr(G)$ is induced by the action 
	on elementary divisors.
	
	We write $\pi:\tbG^*\to\bG^*$ for the surjection induced by the regular
	embedding $\bG\hookrightarrow\tbG$. Here, $\tG^* = (\tbG^*)^F$ is a special
	Clifford group over $\F_q$.
	Let $\tscU(n)$, $\tscU_1(n)$, $\tscU_2(n)$ be the same as $\scU(n)$ $\scU_1(n)$, $\scU_2(n)$ when regarded as the set of maps.
	In $\tscU(n)$, an element $\bm$ is counted $2^{k_\bm}$-times if $\bm\in \tscU_1(n)$ and  $\bm$ is counted $2^{k_\bm-1}$-times if $\bm\in \tscU_2(n)$.
	For a semisimple $2'$-element $s$ of $G^*$, we define 
	$\widetilde{\Upsilon}(s):=\tscU(w_{x-1}(s))\times\prod\limits_{\Gamma\in \cF\setminus \{ x-1,x+1\}}\mathscr P(m_{\Gamma}(\pi(s)))$.
	By the proof of~\cite[Prop.~3.4]{FM20}, $\widetilde{\Upsilon}(s)$ is the labelling set for $\IBr(\cE_2(\tG,\tilde s))$.
	Let $\Upsilon(\tG)$ be the set of $\tG^*$-conjugacy classes of pairs $(\ts,\mu)$, where $\ts$ is a semisimple $2'$-element of $\tG^*$ and $\mu\in\widetilde{\Upsilon}(\pi(\ts))$.
	Then $\Upsilon(\tG)$
	is a labelling set for $\IBr(\tG)$.
	
	If $z\in\ZZ(\tG^*)_{2'}$, then we may regard $\hat z$ as a linear Brauer character of $\tG$.
	If $\widetilde\chi\in\IBr(\tG)$ corresponds to  $(\ts,\mu)$, then
	by the proof of~\cite[Prop.~3.4]{FM20}, 
	$\hat z \widetilde\chi$ corresponds to $(z\ts,\mu)$ and 
	$\widetilde\chi^{F_0}$ corresponds to $({F_0^*}^{-1}(\ts),{F_0^*}^{-1}(\mu))$; i.e., the action of $F_0$ on $\IBr(\tG)$ is induced by the action 
	on elementary divisors.
	
	Let $B=\langle F_1 \rangle$ with $F_1=F_0^e$ and $q_0=p^e$.
	Then $\tC=\tbG^{F_1}= \CSp_{2n}(q_0)$ and $C=\bG^{F_1}=\Sp_{2n}(q_0)$.
	The conditions (ii) and (iii) of Theorem~\ref{criforlie} follow by~\cite[Cor.~4.6]{FM20}, while (i) is obvious.
	Now we consider (iv).
	As in the proof of Theorem~\ref{corr-ord-char}, the $F_1$-stable conjugacy classes of  semisimple elements of $\tG$ are in bijection with the conjugacy classes of semisimple elements of $\tC$.
Using the above combinatorial desciption, 
	the proof is similar to the one of Theorem~\ref{for-type-A}.
\end{proof}

Now we consider good primes for groups of Lie type and define the following hypotheses which are stronger than Hypothesis~\ref{unitri}.

\begin{hyp}\label{unitri-good}
Let	$\bG$ be a simply-connected simple algebraic group in characteristic $p$
	and $F:\bG\to \bG$ a Frobenius  endomorphism
	endowing $\bG$ with an $\F_{q}$-structure so that $\mathbb G$ is also the complete root datum of $(\bG,F)$.	Let $\ell$ be a prime good for $\mathbf G$, $\ell\ne p$ and not dividing the order of 	$(\ZZ(\bG)/\ZZ^\circ(\bG))^F$.
\begin{enumerate}[\rm(i)]
\item 	With a suitable ordering, the decomposition matrix  of $\bG^{F}$ 
	associated with the basic set $\mathcal E(\bG^{F}, \ell')$  is unitriangular.
\item 	With a suitable ordering, the decomposition matrix  of $\widetilde \bG^{F}$ 
associated with the basic set $\mathcal E(\widetilde \bG^{F}, \ell')$  is unitriangular.
\end{enumerate}
\end{hyp}

Note that in Hypothesis~\ref{unitri-good}, by a theorem of Geck--Hiss the union of series corresponding to $\ell$-regular semi-simple elements $\mathcal E(\bG^{F}, \ell')$ (resp. $\mathcal E(\widetilde \bG^{F}, \ell')$) is a basic set of $\bG^F$ (resp. $\tbG^F$); see \cite[Thm.~14.4]{CE04}.

We have

\begin{prop}\label{prop-equiva}
	Hypothesis~\ref{unitri-good} (i) is equivalent to Hypothesis~\ref{unitri-good} (ii).
\end{prop}

\begin{proof}
Note that $\ell\nmid |\tbG^F/\bG^F\ZZ(\tbG^F)|$. If the basic set $\mathcal E(\widetilde \bG^{F}, \ell')$ is unitriangular, then by \cite[Thm.~2.5]{De17}, the basic set $\mathcal E(\bG^{F}, \ell')$ is unitriangular.

Conversely, if the basic set $\cE:=\cE(\bG^{F}, \ell')$ is unitriangular, then according to Corollary \ref{going-up-special}, it suffices to show that $\cE(\widetilde \bG^{F}, \ell')=\Irr(\tbG^F\mid\cE)\cap\Irr(\tbG^F\mid 1_{\ZZ(\tbG^F)_{\ell}})$.
Denote by $\iota:\bG\hookrightarrow \tbG$ the regular embedding, which induces the surjective endomorphism $\iota^*:\tbG^*\to\bG^*$.
Let $s\in {\bG^*}^F$ be a semisimple $\ell'$-element and $\ts\in(\tbG^*)^F$ be the semisimple $\ell'$-element such that $s=\iota^*(\ts)$.
Then $\Irr(\tbG^F \mid  \cE(\bG^F,s)) = \bigcup_{z\in\ZZ((\tbG^*)^F)}\cE(\tbG^F,z\ts)$.
Note that $\cE(\tbG^F,\ts)\subseteq\Irr(\tbG^F\mid 1_{\ZZ(\tbG^F)_\ell})$.
Recall that $\hat z=\Psi_{\tbG,F}(z,1_{(\tbG^*)^F})\in\Irr(\tbG^F/\bG^F)$, then by \cite[Thm.~7.1]{DM90},
 $\cE(\tbG^F,z\ts)=\hat z\otimes \cE(\tbG^F,\ts)$.
Hence
$\cE(\tbG^F,z\ts)\cap \Irr(\tbG^F\mid 1_{\ZZ(\tbG^F)_\ell})\ne\emptyset$ if and only if $z$ is of $\ell'$-order.
So $\Irr(\tbG^F\mid \cE(\bG^F,s))\cap\Irr(\tbG^F\mid 1_{\ZZ(\tbG^F)_\ell})=\bigcup_{z\in\ZZ(\tbG^*)^F_{\ell'}}\cE(\tbG^F,z\ts)$.
Thus we complete the proof.
\end{proof}

We say Hypothesis~\ref{unitri-good} holds for $\mathbb G$ if it is true for any $q$.

Since Hypothesis~\ref{unitri-good} clearly implies Hypothesis~\ref{unitri}, we have the following corollary as an immediate consequence of Theorem~\ref{mainthm}.

\begin{cor}\label{cer-2} 
	Keep the setup of Theorem~\ref{corr-ord-char}. 
	Let $\ell$ be a prime good for $\bG$, $\ell\nmid q$ and not dividing the order of 	$(\ZZ(\bG)/\ZZ^\circ(\bG))^F$.
	Assume that Hypothesis~\ref{unitri-good} holds for $\mathbb G$. 
	If $S:=\bG^F/\ZZ(\bG^F)$ is simple,
	then  the  (iBG)  condition (see Definition~\ref{iBGC}) holds for $S$ and $\ell$.
\end{cor}

By~\cite{GH97},  Hypothesis~\ref{unitri-good} holds when $\mathbb G$ is of classical type and $\ell$ is a linear prime.
Recently, the unitriangularity of decomposition matrix was proven for unipotent blocks and good primes \cite{BDT20}.
But for general cases, this is still open.

We remark that if $\bG$ is not of type $\mathsf A$ and $\ell$ is good for $\bG$, then  $\ell$ does not divide the order of $(\ZZ(\bG)/\ZZ^\circ(\bG))^F$.

Next, we will prove Corollary 4.
Groups $\mathrm{PSp}_4(q)$ are a class of simple groups with abelian Sylow 3-subgroups.

\begin{lem}\label{ellodd}
If both $\ell$ and $q$ are odd and $\ell\nmid q$, then
$\mathcal E(\Sp_4(q),\ell')$ forms a unitriangular basic set of $\Sp_4(q)$ in the sense of Definition~\ref{DefBasic}.
\end{lem}

\begin{proof}
Using the notation of Section~\ref{SG-Lie}, we have $\bG=\Sp_{4}(\overline\F_q)$ and $F=F_0^f$, where $f$ satisfies $q=p^f$.
The irreducible characters of $G=\bG^F$ were computed by~\cite{Sr68} and we will use the parametrizations of characters there.
The semisimple labels in the Jordan decomposition of characters of $G$ are given in Table A1 of~\cite{Wh90b}.

Assume that $\ell$ divides the order of  $G$.
Note that $|G|=q^4(q-1)^2(q+1)^2(q^2+1)$.
Thus $\ell$ divides exactly one of $q-1$, $q+1$ or $q^2+1$.
By~\cite{Wh90a}, any block of $G$ is either of cyclic defect group or of maximal defect if $\ell\mid (q\pm1)$ and
any block of $G$ is of cyclic defect group if $\ell\mid (q^2+1)$.
The decomposition matrices of blocks of $G$ of  maximal defect were studied in~\cite{Wh90a}, while the Brauer trees for cyclic blocks were given in~\cite{Wh92}.
If $\ell\mid (q-1)$, then $\ell$ is a linear prime and then the assertion follows from~\cite{GH97}.
So in the following we assume that $\ell\mid (q+1)$ or $\ell\mid (q^2+1)$.

If $\ell\mid (q+1)$, then using the notation of~\cite{Wh90a}, we choose basic sets
$\{1_G,\theta_{10},\theta_{11},\theta_{12},\}$, $\{\Phi_1,\Phi_2,\Phi_3,\Phi_4\}$,
$\{\theta_3, \theta_4,\Phi_1, \theta_1,\theta_2\}$,
$\{\xi_1(s),\xi'_1(s)\}$, $\{\xi_{21}(s),\xi'_{22}(s) \}$,
$\{\chi_6(s),\chi_7(s)\}$, $\{\chi_4(s,t)\}$
for the blocks of maximal defect $b_0$, $b_1$, $b_2$, $b_1(s)$, $b_{21}(s)$,  $b_{67}(s)$,  $b_{4}(s,t)$  respectively.
Under the notation of~\cite[\S 2]{Wh92}, 
we choose basic sets 
$\{\chi_2(s)\}$, $\{\chi_5(s,l)\}$,
$\{\xi_{21}(s)\}$, $\{\xi_{22}(s)\}$,
$\{ \chi_8(l),\chi_9(l)  \}$,
$\{ \xi_3(l),\xi'_3(l)  \}$, $\{ \xi'_{41}(l),\xi'_{42}(l)  \}$, 
$\{ \Phi_5,\phi_7  \}$,  $\{ \Phi_6,\phi_8  \}$,  $\{ \theta_5,\theta_8  \}$,
$\{ \theta_6,\theta_7   \}$
for the blocks of cyclic defect groups
$B_2(s)$, $B_5(s,l)$, $B_{21}(s)$, $B_{22}(s)$,
$B_{89}(l)$, $B_{53}(l)$, $B_{54}(l)$,  $B_{57}$, $B_{68}$, $B_{58}$, $B_{67}$ respectively.
If $\ell\mid (q^2+1)$ then using the notation of~\cite[\S3]{Wh92},
we choose basic sets 
$\{ 1_G, \theta_9,\theta_{10},\theta_{13}\}$, $\{ \theta_5,\theta_6,\theta_{7},\theta_{8} \}$,
$\{\chi_1(s)\}$
for the blocks
$B_0$, $B_0^*$, $B_1(s)$
respectively.

We denote by $\mathcal E_0$ the union of the basic sets of blocks of $G$ in both cases of the above paragraph.
By~\cite{Wh90a} and~\cite{Wh92}, the decomposition matrix associated with $\mathcal E_0$ is unitriangular.
Note that  by the Brauer tree given in~\cite{Wh92},  the basic sets of cyclic blocks above consist of the non-exceptional characters and then the unitriangular shape is well-known~(see, e.g.,~\cite[Thm.~5.1.2]{Cr19}).
Using the semisimple labels given in Table A1 of~\cite{Wh90b},
we can check directly that $\mathcal E_0=\mathcal E(\bG^F,\ell')$, as desired.
\end{proof}

\begin{cor}\label{sp4}
The  (iBG)  condition holds for the simple group $\mathrm{PSp}_4(q)$.
\end{cor}

\begin{proof}
	Thanks to Theorem~\ref{typeC-2}, we only need to consider odd primes $\ell$. If $q$ is even, then by~\cite[Thm.~2.5.12]{GLS98}, 	
$\mathrm{Out}(S)$ is cyclic and this assertion follows from~\cite[Thm.~4.4]{NST17}.	
Thus we assume further that $q$ is odd.	
Also by~\cite[Thm.~5.1]{NST17}, we can assume $\ell\nmid q$.
	
The universal covering group of the simple group $S=\mathrm{PSp}_4(q)$ ($q\ne2$) is $G=\Sp_4(q)$.
Using the notation of  \S\ref{SG-Lie}, we set $\bG=\Sp_{4}(\overline\F_q)$ and $\widetilde{\bG}=\mathrm{CSp}_{4}(\overline\F_q)$.
If $\ell$ is odd, then 
Hypothesis~\ref{unitri-good} holds by Lemma~\ref{ellodd}.
Thus the assertion follows from
Corollary~\ref{cer-2}.
\end{proof}

As an application, we prove Corollary 4.

\begin{proof}[Proof of Corollary 4.
	]
By Theorem~\ref{main-thm-SV16}, it suffices to prove that any non-abelian simple group $S$ involved in $G$ satisfies the (iBG)  condition.
If $S$ is not of Lie type, then $S$ satisfies the (iBG)  condition for any prime by~\cite[Cor.~4.6]{NST17}.

So we assume that $S$ is of Lie type and then by~\cite[Lemma~2.2]{FLL17}, $S$ is a simple group of type $\mathsf A$, a Suzuki group, or $\mathrm{PSp}_4(q)$ ($q\ne 2$, $3\nmid q$).
Then $S$ satisfies the (iBG) condition by Theorem~\ref{for-type-A},  \cite[Cor.~4.6]{NST17} and Corollary~\ref{sp4}, respectively.
This completes the proof.
\end{proof}

\section{Unitriangular basic sets of finite reductive groups}\label{sec-unitri-Lie}

In this section, we discuss the existence of unitriangular basic sets of finite reductive groups.
Decomposition matrices can be computed blockwise.
As explained before, we view Theorem \ref{TriangUp} as a missing link between Bonnaf\'e--Dat--Rouquier's results relating blocks of finite reductive groups $\bG^F$ to (essentially) unipotent blocks of certain non-connected groups $\bN^F$ (see \cite[Thm.~1.1]{BDR17}) on one hand and Brunat--Dudas--Taylor's theorem about decomposition matrices of unipotent blocks of connected groups (see \cite[Thm.~A]{BDT20}) on the other hand. In order to apply it to more primes we make use of Denoncin's stronger results from \cite{De17} on the groups $\SL_n(q)$ and $\SU_n(q)$.

\begin{prop}\label{BDT+}
Let $\bG$ be a simply-connected simple algebraic group with Frobenius endomorphism $F:\bG\to\bG$ defining it over a finite field of characteristic $p$.
Assume that $\ell$ and $p$ are distinct and both good for $\bG$.
Let $B_1$ be the sum of unipotent $\ell$-blocks of $\bG^F$.
Then the decomposition matrix of $B_1$ is lower unitriangular with respect to a suitable $\Aut(\bG^F)$-stable basic set and some ordering of its rows and columns.
\end{prop}

\begin{proof}
If $\bG$ is of type $\mathsf A$ this follows from \cite[Thm.~A]{De17}. 

In other types Theorem A of \cite{BDT20} applies since the fact that $\ell$ is good for $\bG$ also implies that it does not divide the order of $\ZZ(\bG)/\ZZ^\circ(\bG)$, hence $\ell$ is very good in the sense of \cite{BDT20}.
The basic set there is the set of unipotent characters which is stable under automorphisms (see for instance \cite[Prop.~2.6]{CS13}).
\end{proof}

In the following statement we apply the above result on unitriangular basic sets for unipotent blocks and obtain that many non-unipotent blocks have analogously such a
basic set as well. 
For an analogous statement applying to all $\ell$-blocks it would be needed to know that every isolated $\ell$-block has a sufficiently stable unitriangular basic set (see Proposition~\ref{RedIso} below).

\begin{thm}\label{thm_appl}
Let $\bG$ be a simple simply-connected algebraic group with Frobenius endomorphism $F:\bG\to\bG$ defining it over a finite field of characteristic $p$.
Assume that $\ell$ and $p$ are distinct and both good for $\bG$. 
Let $(\bG^*,F)$ be dual to $(\bG,F)$ and $s\in(\bG^*)^F$ a semisimple $\ell'$-element such that $\C_{\bG^*}^\circ(s)$ is a Levi subgroup of $\bG^*$. Assume moreover that $\C_{\bG^*}(s)/\C_{\bG^*}^\circ(s)$ is cyclic or that $\ell\nmid (q^2-1)$.
Let $B_s$ be the sum of $\ell$-blocks of $\bG^F$ associated to $s$ as in \cite[Thm.~9.12]{CE04}.

Then the $\ell$-decomposition matrix of $B_s$ is lower unitriangular with respect to a suitable basic set.
\end{thm}

\begin{proof}
We apply \cite[1.1]{BDR17} and \cite{Ru20} and get that $B_s$ is Morita equivalent to some $b_s$, a sum of blocks of a group $\bN^F_s$ covering the sum of unipotent blocks of $\bL^F_s$
where $\bL_s$ is an $F$-stable Levi subgroup $\bG$ dual to $\C_{\bG^*}^\circ(s)$
and $\bN_s/\bL_s\cong \C_{\bG^*}(s)/\C_{\bG^*}^\circ(s)$ by duality. 
Note that since $s$ is an element of $\ell'$-order the group $\bN_s/\bL_s$ is an $\ell'$-group, see \cite[13.16~(i)]{CE04}.
Moreover the decomposition matrices of $B_s$ and $b_s$ coincide via the character bijections given by the Morita equivalence, see \cite[\S 2.2, Ex. 3]{Be91}. So there remains to check that $b_s$ has a unitriangular basic set in the sense of Definition \ref{DefBasic}.	
	
For later we first ensure that every character of a unipotent block of $\bL^F_s$ extends to its
stabilizer in $\bN^F_s$. 
It is sufficient to consider the case where $\bN_s^F/\bL_s^F$ is non-cyclic and our character is stable under the whole $\bN_s$, since otherwise extendibility is ensured by the cyclicity of the inertial quotient. 
We may then apply \cite[Prop.~16]{Ru20} to a module where $\bL_s^F$ acts on the right and $\bG^F$ acts trvially on the left. We get that every $\bN_s^F$-stable character of $\bL_s^F$ in a unipotent block extends to $\bN_s^F$.
	 
Let $\bL_0 = [\bL_s,\bL_s]$ and $L_0=\bL_0^F$.
According to \cite[12.14]{MT11} the group $\bL_0$ is the direct product
of simply-connected simple groups and hence $L_0$ is the direct product of groups $\bH^{F'}$, where $\bH$ is a simply-connected simple algebraic group and $F' : \bH \to \bH$ is a Frobenius endomorphism. 
We can apply Proposition \ref{BDT+} to the unipotent $\ell$-blocks of those $\bH^{F'}$. 
The unipotent characters of $\bL^F_s$ and $L_0$ correspond by restriction, so  the unipotent blocks of $\bL^F_s$
are covering only unipotent blocks of $\bH^{F'}$. 
Being good for $\bG$, the prime $\ell$ is also good for $\bL_s$ (see \cite[13.10]{CE04}) and hence for each $\bH$. Accordingly the sum of unipotent blocks of each $\bH^{F'}$ has a
unitriangular $\Aut(\bH^{F'})$-stable basic set by Proposition \ref{BDT+}. 
Thanks to the description in terms of roots and Weyl group one sees easily that the action of $\C_{\bG^*}(s)$ on  $\C_{\bG^*}^\circ(s)$ permutes the irreducible components of its root system, so dually $\bN_s$ permutes only factors of $\bL_0$ that are of the same type. This implies that $\bN_s^F$-conjugacy permutes factors $\bH^{F'}$ of $L_0$ of the same type, hence permutes the basic sets chosen by applying Proposition \ref{BDT+}. Consequently there exists an $\bN^F_s$-stable unitriangular basic set $\cB_0$ for the sum of unipotent blocks of $L_0$. 
The quotient $\bL^F_s/L_0$ being abelian there exists a unique subgroup $L_1$ with $L_0\le L_1\le \bL^F_s$
such that $|\bL_s^F/L_0|_\ell=|\bL_s^F/L_1|$. 
According to Theorem \ref{TriangUp} the set $\cB_1 := \Irr(L_1\mid\cB_0)$ forms an $\bN_s^F$-stable unitriangular basic set of the sum of blocks of $L_1$ covering a unipotent block of $L_0$. 
Note that maximal extendibility holds with respect
to $L_0\unlhd  \bL^F_s$, see \cite[15.11]{CE04}. Let $\Lambda$ be now an extension map with respect to $L_1\unlhd \bL_s^F$
for $\cB_1$. 
This extension map can be chosen to be $\bN^F_s$-equivariant by an application of \cite[11.31]{Is76}, since $|\bN_s^F/\bL_s^F|$ is coprime to $|\bL_s^F/L_1|$ and hence every character of $L_1$ extends to its stabilizer in $\bN^F_s$. 
By another application of Theorem \ref{TriangUp} (see Remark \ref{remTri} (ii)) we see that 
$\{\Ind^{\bL_s^F}_{(\bL_s^F)_\chi}(\Lambda(\chi))\mid\chi\in\cB_1 \}$ 
is a unitriangular basic set of $\bL^F_s$. It is clearly $\bN^F_s$-stable, since $\Lambda$ is $\bN^F_s$-equivariant and $\cB_1$ is $\bN_s^F$-stable. 
Via Theorem \ref{TriangUp} we obtain a unitriangular basic set of $\bN^F_s$ for all blocks covering a unipotent block of $L_0$.
In particular we see that $b_s$ has a unitriangular basic set. As said before, this completes our proof.
\end{proof}

Note that the basic set obtained by the above construction might not be $\Aut(\bG^F)_{B_s}$-stable in general. 
For several specific groups another construction is possible and provides an $\Aut(\bG^F)_{B_s}$-stable basic set as explained in the following proposition.

\begin{prop}\label{appl_C} Let $\bG$ be a simply-connected simple algebraic group defined over a finite field of good characteristic $p$ with associated Frobenius endomorphism $F\colon \bG\to\bG$. Let $\ell$ be an odd prime, $\ell\neq p$. Let $(\bG^*,F)$ be dual to $(\bG,F)$ and assume $s\in\bG^*{}^F$ is a semisimple $\ell'$-element such that $\C_{\bG^*}^\circ(s)$ is a \textbf{Levi subgroup} of $\bG^*$. Let $B_s$ be the sum of $\ell$-blocks of $\bG^F$ associated with $s$.
	Assume moreover one of the following hypotheses.
	\begin{enumerate}[(i)]
		\item $\bG$ has  type $\mathsf B$, $\mathsf C$ or $\mathsf D$.

	\item	$\bG$ has  type $\mathsf E_7$ or $\mathsf E_8$, and $\ell\geq 11$.

		\item	$\bG$ has  type $\mathsf E_6$ or $\mathsf F_4$, and $\ell\geq 7$.

		\item	$\bG$ has  type $\mathsf G_2$ and $\ell\geq 5$.
		
		\end{enumerate}

Then the rational series $\cE(\bG^F, s)$ is a unitriangular basic set for $B_s$ that is $\Aut(\bG^F)_{B_s}$-stable.
\end{prop}

\begin{proof} Note that automorphisms of $\bG^F$ stabilize the set $\cE(\bG^F,\ell ')$ as can be seen from the action on generalized characters R$_\bT^\bG(\theta)$ (see \cite[9.2]{DM90}). It is then clear that automorphisms of $\bG^F$ that stabilize $\Irr(B_s)$ will also preserve $\cE(\bG^F,\ell ')\cap\Irr(B_s)=\cE(\bG^F,s)$. So we now concentrate on the first part of our claim.
	
Recall $\bL_s$ from the above proof, a Levi subgroup of $\bG$ dual to $\C_{\bG^*}^\circ(s)$. We abbreviate $\bL=\bL_s$.

Let us start with the case of type $\mathsf C$ where $\bG^F$ is a symplectic group Sp$_{2n}(q)$ ($n\geq 2$).
Thanks to standard calculations (see for instance \cite[p.~126]{FS89}) $\bL^F$
is a direct product of groups $H_i$ where each $H_i$ is either isomorphic to some $\GL_{n_i}(q_i)$, $\GU_{n_i}(q_i)$ or to some $\Sp_{2n'}(q)$ for $q_i$ a power of $q$ and $n_i, n'\le n$. 
Every unipotent block of $\bL^F$ is hence the direct product of unipotent blocks of these factors.

The unipotent characters of $H_i$ form an $\Aut(H_i)$-stable unitriangular basic set by \cite[Cor. B]{Ge91} for the groups $\GL_{n_i}(q_i)$, $\GU_{n_i}(q_i)$ and by Proposition \ref{BDT+} in the other case. 
Hence the unipotent characters $\cE(\bL^F,1)$ of $\bL^F$ form a unitriangular basic set for $B_1(\bL^F)$
the sum of unipotent blocks of $\bL^F$. Recall from the proof of Theorem~\ref{thm_appl} the sum $b_s$ of blocks of $\bN_s^F$ covering $B_1(\bL^F)$. Since $b_s$ is $\bN_s^F$-stable, one has $\Irr(b_s)=\Irr(\bN_s^F\mid\Irr(B_1(\bL^F)))$,  $\IBr(b_s)=\IBr(\bN_s^F\mid\IBr(B_1(\bL^F)))$ and Corollary \ref{cor} implies that
$\Irr(\bN^F_s\mid \cE(\bL^F,1))$ forms a unitriangular basic set for $b_s$. 
The Bonnaf\'e--Dat--Rouquier Morita equivalence (cf. Example 7.9 of \cite{BDR17}) maps the rational Lusztig series $\cE(\bG^F ,s)$ to $\Irr(\bN^F_s\mid \cE(\bL^F,1))$. This completes our proof in that case.

Let us assume now that $\bG$ has type $\mathsf{B}$ or $\mathsf{D}$, so that $\bG^F=\Spin(V)$ for $V$ an orthogonal space over $\F_q$. Let $\tG:=\tbG^F$ be the corresponding special Clifford group over $V$.
If $\iota:\bG\hookrightarrow\tbG'$ is the regular embedding, then $\tbG\le \tbG'$; and $\tbG=\tbG'$ unless $\bG$ is of type ${\mathsf D}$.
Using \cite[Thm.~1.1]{BDR17} as in the above proof for type $\mathsf{C}$, it suffices to show that $\cE(\bL^F,1)$ forms a unitriangular basic set for the sum of unipotent blocks of $\bL^F$. 
Let $\tbL=\bL \ZZ(\tbG)$.
Note that $\iota$ induces a surjective morphism $\iota^*_0:{\tbG}^*\to\bG^*$ and we let $\ts\in(\tbG^*)^F$ be a semisimple $\ell'$-element such that $\iota^*_0(\ts)=s$ and $\ts\in\tbL^*$.
By Corollary \ref{going-up-special} (which is similar to the proof of Proposition \ref{prop-equiva}),  $\cE(\bL^F,1)$ is a unitriangular basic set if and only if $\cE(\tbL^F,1)$ is a unitriangular basic set since $|\tbL^F/\bL^F\ZZ(\tbG^F)|\le 2$.

Let $\tG_0:=\mathrm{SO}(V)$.
Then there is a natural epimorphism $\pi:\tG\to\tG_0$. Let $\tL_0=\pi(\tbL^F)$.
Note that each unipotent character of $\tbL^F$ has $\ZZ(\tG)$ in its kernel. 
Therefore, $\cE(\tbL^F,1)$ can be identified with $\cE(\tL_0,1)$, which is a basic set of the the sum of unipotent blocks of $\tbL^F$ (and of $\tL_0$).
So it suffices to show that $\cE(\tL_0,1)$ form a unitriangular basic set for the sum of unipotent blocks of $\tL_0$. 

Through easy calculations (see also \cite[\S3]{FS89}) we see that $\tL_0$
is a direct product of groups $H_i$ where each $H_i$ is either isomorphic to some
general linear or unitary groups over $\F_{q_i}$ (with $\F_{q_i}\supseteq \F_{q}$) and a special orthogonal group over $\F_q$. 
Every unipotent block of $\tL_0$ is hence the direct product of unipotent blocks of these factors.
Thus our claim holds by the same argument as in the case of type $\mathsf{C}$. This finishes the proof for the whole case (i).

We now turn to the exceptional types under assumptions (ii)--(iv).
As above, we content ourselves with showing that $\cE(\bL^F,1)$ forms a unitriangular basic set for the sum of the unipotent blocks of $\bL^F$.
As observed in the proof of Theorem~\ref{thm_appl} the group $[\bL,\bL]$ is the direct product of simply-connected simple groups and the Dynkin diagram of $[\bL,\bL]$ is a subdiagram of the Dynkin diagram of $\bG$.
Then as before
$L_0=[\bL,\bL]^F$ is the direct product of groups $\bH^{F'}$, where $\bH$ is of simply connected type for which $\ell$ is good under our assumptions.
In addition, if $\bH$ is of type $\mathsf A$, then $\ell>\mathrm{rank}(\bH)+1$ and thus $\ell\nmid |\ZZ(\bH^{F'})|$.
So $\cE(L_0,1)$ forms a unitriangular basic set for the sum of the unipotent blocks of $L_0$ by \cite[Thm.~A]{BDT20}.
Since $\ell\nmid |\bL^F/L_0\ZZ(\bL^F)|$, Corollary \ref{going-up-special} implies that $\cE(\bL^F,1)$ forms a unitriangular basic set for the sum of the unipotent blocks of $\bL^F$.

This completes our proof.	
\end{proof}

\begin{rmk}\label{RemIBr}
 In view of the criterion from \cite{BS20} of the inductive blockwise Alperin weight condition from \cite{Sp13} the above shows that, for many blocks $B$ of finite quasi-simple groups $G$ of Lie type, $\IBr(B)$ is $\Aut(G)_B$-permutation
isomorphic to a subset of $\Irr(G)$.
This is needed in \cite{FLZ21} and \cite{Li21} in the verification of the inductive blockwise Alperin weight condition for types $\mathsf B$ and $\mathsf C$.

More generally, with the assumptions of 
Proposition~\ref{appl_C}, it is easy to see that $\IBr(B_s(G))$ satisfies assumption (ii) of Theorem~\ref{criforlie}. Indeed for $\phi\in\IBr(B_s(G))$ there is some $f^{-1}(\phi)$ in our basic set associated through an $\Aut(G)_{B_s(G)}$-equivariant bijection $f$. The stabilizer property of Theorem~\ref{criforlie}(ii) is then a consequence of the same property satisfied by $\chi:=f^{-1}(\phi)$. After a suitable $\tG$-conjugation we may now assume that $\phi$ and $\chi$ satisfy $(\tG D)_\phi =\tG_\phi D_\phi=(\tG D)_\chi =\tG _\chi D_\chi$, and we have to check that $\phi$ extends to $G D_\phi=G D_\chi$. We know that $\chi$ extends into $\widetilde{\chi}\in \Irr(G D_\phi)$ while ${\chi}^0$ has $\phi$ as an irreducible component with multiplicity 1 by the unitriangularity property. Then $\widetilde{\chi}^0$ has a constituent $\widetilde{\phi}\in\IBr(G D_\phi)$ whose restriction to $G $ has $\phi$ as constituent with multiplicity 1. But then by Clifford theorem $\Res_G^{\tG}\widetilde{\phi}=\phi$ since $\phi$ is $G D_\phi$-invariant.
Note that this is in fact used in the proof of \cite[Prop.~8.1]{FLZ21a}.
\end{rmk}

\medskip

As Conjecture~\ref{GeckC} we have recalled Geck's conjecture that all $\ell$-blocks of (quasi-simple) groups of Lie type of characteristic $\neq\ell$ have a unitriangular decomposition matrix. 
We consider the following weaker conjecture. Recall that a semi-simple element $s$ of a reductive group $\bf K$ is called {\it isolated} if $\C_{\bf K}^\circ(s)$ is included in no proper Levi subgroup of $\bf K$, or equivalently $\ZZ^\circ(\C_{\bf K}^\circ(s))\subseteq \ZZ(\bf K)$.

We keep a prime number $\ell$.

\medskip\noindent{\bf Conjecture~3.}\ {\it 
	For every simply-connected simple group $\bG$ in characteristic $\not=\ell$ with Frobenius
	endomorphism $F$ and every semi-simple isolated element $t \in (\bG^*)^{F}_{\ell '}$, the sum of $\ell$-blocks $B_t$ of $\bG^{F}$ associated to $t$ has an $\Aut(\bG^{F})_{B_t}$-stable unitriangular basic set. }

\medskip
With the considerations of the proof of Theorem \ref{thm_appl} we show that this conjecture implies Geck's more general conjecture in the cases where \cite{BDR17} and \cite{Ru20} apply.

\begin{prop}\label{RedIso} Assume Conjecture 3.
Let $\bG$ be a simply-connected simple algebraic group with Frobenius endomorphism $F : \bG\to \bG$ defining an $\F_q$-structure on $\bG$ and $s\in (\bG^*)^F$ be a semisimple $\ell'$-element. 
If $\C_{\bG^*}(s)/\C_{\bG^*}^\circ(s)$ is cyclic or $\ell\nmid (q^2-1)$, then the sum of blocks of $\bG^F$ associated to $s$ has a unitriangular basic set.	
\end{prop}

\begin{proof}
The proof of Theorem \ref{thm_appl} can be adapted to the sum of blocks associated to $s$. We recall the Levi subgroup $\bL_s^*=\C_{\bG^*}(\ZZ^\circ(\C_{\bG^*}^\circ(s)))$ where the element $s$ is isolated. We still denote by $\bL_s$ a dual and $\bL_0:=[\bL_s,\bL_s]$.
Let $c_0$ be the block of $\bL_0^F$ covered by $c_s$, the sum of blocks of $\bL_s^F$ associated to $s$. 
Let $\pi$ be the corresponding epimorphism $\bL^*_s\to\bL^*_0$.
We denote by $\bH_i$ the product of simply-connected simple groups that are permuted transitively by $F$.
The decomposition of $\bL_0$ into a direct product of simply-connected groups implies that analogously a quotient of $\bL_s^*$ is a central product of adjoint groups. 
The element $\pi(s)$ accordingly decomposes along those factors and can be written as product of the elements $t_i\in \bH_i^*{}^F$. The block $c_0$ is the product of blocks $c_i$ of $\bH_i^F$, that corresponds to elements $t_i$. Each $t_i$ is isolated in $\bH_i^*$ since connected centralizers correspond through $\pi$ (see for instance \cite[13.13(iii)]{CE04}). Hence by our assumption $c_i$ has an $\Aut(\bH^F_i)_{c_i}$-stable unitriangular basic set.

The rest of the proof of Theorem~\ref{thm_appl} applies without altering and thereby shows the statement.
\end{proof}

\section*{Acknowledgements}  
  
The first author would like to express his deep thanks to Jiping Zhang for his support and numerous suggestions.
The second author would like to thank Gabriel Navarro for some fruitful discussions on basic sets. 
Finally both authors are grateful to Gunter Malle and the referees for their careful reading, remarks and suggestions that led to greatly improve this paper.

\end{document}